\documentclass[11pt,letterpaper]{article}
\usepackage[utf8]{inputenc}
\usepackage[margin=1.0in]{geometry}
\usepackage{setspace}
\usepackage{alphalph}

\usepackage{amsmath}
\usepackage{amsthm}
\usepackage{amssymb}
\usepackage{mathrsfs}
\usepackage{graphicx}
\usepackage{dsfont}
\usepackage{enumerate}
\usepackage{mathabx}
\usepackage{mathdots}
\usepackage{authblk}
\usepackage{titlesec}
\usepackage{soul}
\usepackage{tikz-cd}
\usepackage{mathtools}
\usepackage{slashed}

\usepackage{hyperref}
\hypersetup{
    colorlinks=true,
    urlcolor=cyan,
    linkcolor=blue,
    citecolor=blue,
}

\newcounter{bigthm}
\newtheorem{bigtheorem}[bigthm]{Theorem} % Define el entorno

\newtheoremstyle{citedtheorem}%
  {3pt}% (space above)
  {3pt}% (space below)
  {\itshape}% (body font)
  {}% (indent amount)
  {\bfseries}% {theorem head font}
  {.}% {punctuation after theorem head}
  {.5em}% {space after theorem head}
  {\thmname{#1} \thmnumber{#2} \thmnote{\normalfont#3}}% {theorem head spec}

\newtheoremstyle{citeddefinition}%
  {3pt}% (space above)
  {3pt}% (space below)
  {\normalfont}% (body font)
  {}% (indent amount)
  {\bfseries}% {theorem head font}
  {.}% {punctuation after theorem head}
  {.5em}% {space after theorem head}
  {\thmname{#1} \thmnumber{#2} \thmnote{\normalfont#3}}% {theorem head spec}

\newtheoremstyle{citedexample}%
  {3pt}% (space above)
  {3pt}% (space below)
  {\normalfont}% (body font)
  {}% (indent amount)
  {\itshape}% {theorem head font}
  {.}% {punctuation after theorem head}
  {.5em}% {space after theorem head}
  {\thmname{#1} \thmnumber{#2} \thmnote{\normalfont#3}}% {theorem head spec}

\theoremstyle{plain}
\newtheorem{theorem}{Theorem}[section]
\newtheorem{lemma}[theorem]{Lemma}
\newtheorem{corollary}[theorem]{Corollary}
\newtheorem{proposition}[theorem]{Proposition}
\newtheorem{definition}[theorem]{Definition}

\theoremstyle{citedtheorem}

\theoremstyle{definition}

\theoremstyle{citeddefinition}

\theoremstyle{remark}
\newtheorem{remark}[theorem]{Remark}

\theoremstyle{citedexample}

\usepackage{cleveref}

\newcommand{\cD}{\mathcal{D}}

\newcommand{\cH}{\mathcal{H}}

\newcommand{\cO}{\mathcal{O}}

\newcommand{\sB}{\mathscr{B}}

\newcommand{\sN}{\mathscr{N}}

\newcommand{\sS}{\mathscr{S}}

\newcommand{\sV}{\mathscr{V}}

\renewcommand{\a}{\alpha} % \a gave ???
\renewcommand{\b}{\beta} % \b gave bars under letters

\newcommand{\g}{\gamma}

\newcommand{\e}{\varepsilon}

\newcommand{\m}{\mu}
\newcommand{\n}{\nu}
\newcommand{\x}{\xi}

\newcommand{\vp}{\varphi}

\renewcommand{\l}{\ell} % \l gave produces Polish l

\newcommand{\R}{\mathbb{R}}

\renewcommand{\d}{\partial} % \d already defined, produces ???

\newcommand{\sub}{\subseteq}
\newcommand{\setm}{\, \setminus \,}

\newcommand{\8}{\infty}

\DeclareMathOperator{\End}{End}
\DeclareMathOperator{\id}{id}

\let\Im\relax
\DeclareMathOperator{\Im}{Im}

\DeclareMathOperator{\Spec}{Spec}

\newcommand{\cc}{\subset\subset}
\newcommand{\loc}{\mathrm{loc}}

\DeclareMathOperator{\curl}{curl}

\newcommand{\hk}{\mathbin{\! \hbox{\vrule height0.3pt width5pt depth 0.2pt \vrule height4pt width0.4pt depth 0.2pt}}}
\DeclareMathOperator{\Ric}{Ric}

\title{A Calder\'{o}n Problem for Beltrami Fields}
\author{Alberto Enciso and Carlos Valero}

\begin{document}

\maketitle

\begin{abstract}
    On a $3$-dimensional Riemannian manifold with boundary, we define an analogue of the Dirichlet-to-Neumann map for Beltrami fields, which are the eigenvectors of the curl operator and play a major role in fluid mechanics. This map sends the normal component of a Beltrami field to its tangential component on the boundary. In this paper we establish two results showing how this normal-to-tangential map encodes geometric information on the underlying manifold. First, we show that the normal-to-tangential map is a pseudodifferential operator of order zero on the boundary whose total symbol determines the Taylor series of the metric at the boundary. Second, we go on to show that a real-analytic simply connected $3$-manifold can be reconstructed from its normal-to-tangential map. Interestingly, since Green's functions do not exist for the Beltrami field equation, a key idea of the proof is to find an appropriate substitute, which turn out to have a natural physical interpretation as the magnetic fields generated by small current loops. 
    \end{abstract}

\tableofcontents

\newpage

\section{Introduction}\label{sec: intro} 

A fundamental question in the field of inverse problems is the Calder\'{o}n problem, which was first introduced by Alberto Calder\'{o}n in \cite{calderon1980}, where he considers the problem of recovering the electrical conductivity $\gamma$ of a material from measurements of electric potential and induced currents on the boundary. By allowing the conductivity to be anisotropic, and thus replacing the scalar function $\gamma$ with a positive symmetric $2$-tensor $\gamma^{ij}$, one can recast this problem in the language of differential geometry as follows: can a Riemannian manifold $(M,g)$ with boundary $\d M$ be recovered up to isometry from the set of Cauchy data of harmonic functions at the boundary? The Riemannian metric thus corresponds physically to the conductivity, and the Cauchy data of harmonic functions corresponds to the measurements of potential and current on the boundary. More generally, one can replace harmonic functions by solutions to the Helmholtz equation with frequency~$\lambda$, with the original Calder\'on problem corresponding to the case $\lambda=0$. The Cauchy data is conveniently captured by the Dirichlet-to-Neumann map at frequency~$\lambda$, $\Lambda_{g,\lambda} : C^\8(\d M) \to C^{\8}(\d M)$, which sends a function $f \in C^\8(\d M)$ on the boundary to $\d_n \varphi |_{\d M}$, where $\varphi \in C^\8(M)$ is the only solution to the equation $\Delta_g \varphi+\lambda \varphi = 0$ in~$M$ with Dirichlet datum $\varphi|_{\d M} = f$. 

The purpose of the present paper is to first introduce and study an analogue of the Calder\'{o}n inverse boundary problem for Beltrami fields. Beltrami fields have been a focus of research since the 19th century. Known as force-free fields in magnetohydrodynamics, these fields represent configurations where the magnetic field is aligned with its own curl, resulting in no net force. In fluid mechanics, Beltrami fields are especially significant as they constitute a fundamental class of stationary solutions to the three-dimensional incompressible Euler equations. Their distinguished role in non-laminar steady fluid flows is laid bare in the celebrated structure theorem of Arnold~\cite{Ar65}. Beltrami fields on Riemannian manifolds have also received much attention~\cite{Paco, Cardona2023, Tao}, partly due to the fact that nonvanishing Beltrami fields are related to Reeb fields of contact structures.

On an oriented $3$-dimensional Riemannian manifold $M$ with or without boundary, the equation for a Beltrami field of frequency~$\lambda$ is the following system of first-order partial differential equations:
\begin{equation}\label{Beltrami equation intro, vector}
    \curl_g{v} - \lambda v = 0.
\end{equation}
For concreteness, let us assume that $\lambda \in \R \setm \{0\}$. As we shall see below, there is a natural boundary value problem for the Beltrami field equation, and thus a corresponding analogue~$\Sigma_\lambda$ of the Dirichlet-to-Neumann map that conveniently captures the Cauchy data of Beltrami fields on the boundary. Physically, $\Sigma_\lambda$ maps the normal component of the fluid velocity field~$u$ on the boundary to its tangential component. Our main result is that this normal-to-tangential map determines all normal derivatives of the metric at the boundary, and enables us to  reconstruct a real-analytic Riemannian manifold from the given boundary data under suitable hypotheses.

\subsection{The anisotropic Calder\'on problem}

There is an abundant literature on the Calder\'{o}n problem, which remains a key model for many other interesting inverse boundary problems both in differential geometry and in more applied settings such as tomography.

In dimension~2, on compact, connected surfaces, the anisotropic Calder\'on conjecture in the smooth case has been proven~\cite{Lee1989,lassas2001}. Although we will not elaborate on this point, this is true even with only local data. 

In higher dimensions \(n \geq 3\), the problem is still open, and a full answer is only known in the real-analytic category. Specifically, the first important result for the Calder\'on problem, derived by Lee and Uhlmann in the classical work~\cite{Lee1989}, is that the Dirichlet-to-Neumann map is a pseudodifferential operator of order $1$ whose symbol determines the Taylor series of a smooth metric at the boundary. Building on this, Lassas and Uhlmann proved in the early 2000s  that a complete real-analytic Riemannian manifold of any dimension $n \geq 3$ can be recovered from its Dirichlet-to-Neumann map~\cite{lassas2001, lassas2003}. 

Since we will need to revisit their proof, let us briefly comment on the proof of this metric determination result. In \cite{lassas2001}, the authors the authors use the results of~\cite{Lee1989} to show that the Dirichlet-to-Neumann map determines the Green's functions in an extended manifold. They then use the Green's functions to construct a sheaf, from which the manifold may be obtained as a suitable quotient. This approach is motivated by the idea that the Green's functions provide an atlas of real-analytic coordinates for the manifold as one of their arguments varies. In \cite{lassas2003}, a different approach is used, still based on the real-analyticity of the Green's functions, to extend the reconstruction result to complete manifolds with boundary. Here, the authors use the Green's functions to first embed an extended manifold into a suitable Sobolev space, and thence recover the manifold. 

This latter approach seems easier to adapt to systems and other settings than that of \cite{lassas2001}, and indeed has been used in \cite{Krupchyk2011} to recover a complete Riemannian manifold from the Dirichlet-to-Neumann map of its Hodge Laplacian acting on differential forms; in \cite{Gabdurakhmanov_Kokarev_2024} to recover a real-analytic Riemannian manifold, vector bundle, and connection from the Cauchy data of the connection Laplacian acting on sections; and in \cite{lassas2022} to recover the conformal class of a real-analytic Riemannian manifold from the Cauchy data of its conformal Laplacian. In fact, the Green's function approach of \cite{lassas2003, Krupchyk2011, Gabdurakhmanov_Kokarev_2024} allows one to recover a real-analytic metric from knowledge of the corresponding Dirichlet-to-Neumann map on a proper open subset $\Gamma$ of the boundary only; this is sometimes known as the Calder\'{o}n problem with local data. Positive results have also been established for compact connected Einstein manifolds~\cite{guillarmou2009einstein}.

With smooth metrics, the anisotropic Calder\'on conjecture remains an open problem. Notable uniqueness results exist for conformally transversally anisotropic manifolds~\cite{dos2009limiting, Ferreira2016,kenig2013calderon}. Building on previous results on local problems~\cite{DKN2018,DKN2019,DKN2020,DKN2021}, Daud\'e, Helffer, Kamran and Nicoleau~\cite{DKN2024} have recently shown that the anisotropic Calder\'on problem at nonzero frequency admits pairs of non-isometric $C^k$ metrics which give rise to the same Dirichlet-to-Neumann map.

Several significant studies address the Calder\'on problem for singular conductivities. In dimension \(n = 2\), Astala and P\"aiv\"arinta demonstrated~\cite{astala2005calderon,astala2006calderon,astala2016borderlines} that an elliptic isotropic conductivity in \(L^{\infty}(\Omega)\) is uniquely determined by the  Dirichlet-to-Neumann map. In dimension \(n \geq 3\), Caro and Rogers~\cite{caro2016global} proved uniqueness in the global Calder\'on problem for elliptic Lipschitz isotropic conductivities. For local data, Krupchyk and Uhlmann~\cite{krupchyk2016calderon} showed that an isotropic conductivity with three derivatives is uniquely determined by a Dirichlet-to-Neumann map measured on even a very small boundary subset. There are also counterexamples to uniqueness, such as those by Greenleaf, Kurylev, Lassas, and Uhlmann~\cite{greenleaf2003nonuniqueness,greenleaf2009invisibility} for metrics that become degenerate along a closed hypersurface.

\subsection{A Calder\'on problem for Beltrami fields}

Let us now present a Calder\'on-type problem for Beltrami fields, corresponding to Equation~\eqref{Beltrami equation intro, vector}. As above, let $(M,g)$ be an oriented $3$-dimensional Riemannian manifold with boundary $\d M$.  By the canonical isomorphism $TM \to T^*M$ induced by the metric, we may also view vector fields on $M$ as $1$-forms, for which the Beltrami field equation~\eqref{Beltrami equation intro, vector} becomes
\begin{equation}\label{Beltrami equation intro, 1-form}
    * du - \lambda u = 0,
\end{equation}
where $* $ is the Hodge star operator (which maps $k$-forms on the 3-manifold~$M$ to $(3-k)$-forms) and $d$ is the exterior derivative. We shall often make use of this identification throughout the paper.

 To formulate the problem, it is convenient to start by recalling some fundamental properties and results concerning differential forms on a manifold with boundary. First, let us recall that the metric induces a decomposition of the tangent bundle on the boundary,
\begin{equation*}
    TM|_{\d M} = T(\d M) \oplus NM,
\end{equation*}
where $T(\d M)$ is the tangent bundle to the boundary, and $NM$ is a trivial bundle of rank $1$ consisting of vector fields normal to the boundary. An analogous decomposition holds for the cotangent bundle,
\begin{equation*}
    T^*M|_{\d M} = T^*(\d M) \oplus N^*M.
\end{equation*}
The inward unit normal $\nu$ defines a global frame for $NM$, while the $1$-form $\nu^{\flat}$ corresponding to $\nu$ under the canonical isomorphism defines a global frame for the conormal bundle $N^*M$.  Hence, the restriction of any $u \in \Omega^1(M)$ to the boundary can be decomposed into a component tangential to $\d M$, and a component normal to $\d M$. We may therefore write 
\begin{equation}\label{decomposition on the boundary}
    u|_{\d M} = u_t + f \n^{\flat},
\end{equation}
where $u_t$ is a $1$-form tangent to the boundary. We shall write $\nu \hk u|_{\d M}$ for the evaluation of $u|_{\d M}$ on the unit normal $\nu$. Thus, if $u|_{\d M}$ decomposes as in \eqref{decomposition on the boundary}, then we have $\n \hk u|_{\d M} = f$. 

Next, we want to recall the Hodge decomposition for a Riemannian manifold $M$ with boundary $\d M$. To this end, we make the following definitions:
\begin{equation*}
    \Omega^k_{\mathrm{D}}(M) := \{ u \in \Omega^k(M) \ | \ u_t = 0 \ \text{on} \ \d M \},
\end{equation*}
\begin{equation*}
    \Omega^k_{\mathrm{N}}(M) := \{ u \in \Omega^k(M) \ | \ \n \hk u = 0 \ \text{on} \ \d M \}.
\end{equation*}

We also recall~\cite{CappellDeTurckGluckMiller+2006+923+931} that a {\em harmonic field} $h$ is a $k$-form satisfying $dh = 0$ and $d^*h = 0$. Observe that this is in contrast with {\em harmonic forms}, which satisfy $\Delta_d h =0$. Every harmonic field is a harmonic form, but the converse is not true on manifolds with boundary.

We define $\cH^k(M)$ to be the space of harmonic $k$-fields on $M$, and $\cH^{t, k}(M) \sub \cH^k(M)$ to be the space of harmonic $k$-fields that are tangent to the boundary. We then have the following Hodge decomposition theorem for $k$-forms on a manifold with boundary \cite[\S 2.4]{schwarz2006hodge}:

\begin{theorem}[Hodge decomposition]\label{Hodge decomposition theorem}
    Let $M$ be a Riemannian manifold with boundary $\d M$ as above. Then we have the following $L^2$-orthogonal decomposition:
    \begin{equation*}
        \Omega^k(M) = d \Omega_D^{k-1}(M) \oplus \cH^k(M) \oplus\, d^* \Omega_N^{k+1}(M).
    \end{equation*}
    Moreover, $\cH^k(M)$ itself decomposes orthogonally into
    \begin{equation}\label{decomposition of H into H^t and im d}
        \cH^k(M) = \cH^{t,k}(M) \oplus\, [\Im{d} \cap \cH^k(M)].
    \end{equation}
\end{theorem}

In the case that interests us, $k = 1$, and so we shall often drop the index $k$. For $u \in \Omega^1(M)$, we shall let $\cH_M(u)$ and $\cH_M^t(u)$ denote its projection onto the spaces $\cH(M)$ and $\cH^t(M)$, which we call the {\em harmonic part of $u$}, and the {\em harmonic part of $u$ tangent to the boundary}, respectively.  

We are now ready to introduce the boundary value problem for Beltrami fields that will concern us. We shall let $C^\infty_*(\d M)$ denote the space of all zero-mean smooth functions:
\begin{equation*}
    C^\infty_*(\d M) := \left\{ f \in C^\infty(\d M) \ \bigg| \ \int_{\d M} f = 0 \right\}.
\end{equation*}
We define the zero-mean Sobolev spaces $H^s_*(\d M)$ in an analogous way. Then we have the following theorem, which is a consequence of known existence results in the literature on Beltrami fields, such as \cite[Proposition 6.1]{enciso2015}, and standard elliptic estimates, see e.g. \cite[\S 5.11]{Taylor2011pde1}.

\begin{theorem}\label{Theorem on existence of Beltrami fields source}
    Let $w$ be an $H^k$ 1-form over $M$ that is divergence-free. For any~$\lambda$ which is not in certain countable subset $T$ of the real line without accumulation points, the problem
    \begin{equation}\label{Beltrami source problem}
    \begin{cases}
        *du - \lambda u = w, \\
        \n \hk u|_{\d M} = 0,
    \end{cases}
    \end{equation}
    has a unique $H^{k+1}$ solution $u$ with zero harmonic part tangent to the boundary for any $H^k$ 1-form over $M$ that is divergence-free.
\end{theorem}

We shall call $T$ the set of {\em Beltrami singular values}\footnote{The set of Beltrami singular values is the union of the spectrum of the self-adjoint operator defined by the curl operator $*d$ with the tangency boundary condition~\cite{Yoshida1990}, whose domain is dense in the space of $L^2$ co-closed 1-forms, and of the (positive and negative) square roots of the Dirichlet eigenvalues of the Hodge Laplacian. For the purposes of Theorems~\ref{Theorem on existence of Beltrami fields source} and~\ref{Theorem on existence of Beltrami fields}, one could just take the former spectrum, but later on we will use that $\lambda^2$ is not in the Dirichlet Hodge spectrum either.}. As a corollary of Theorem \ref{Theorem on existence of Beltrami fields source}, we have the following:

\begin{theorem}\label{Theorem on existence of Beltrami fields}
    For any real constant $\lambda$ that is not a Beltrami singular value, the boundary value problem 
    \begin{equation} \label{Beltrami boundary problem}
    \begin{cases}
    * d u - \lambda u = 0, \\
    \n \hk u|_{\d M} = f, \\
    \cH^t_M(u) = 0,
    \end{cases}
\end{equation}
has a unique $H^{k+1}$ solution $u$ for all $f \in H^{k + \frac{1}{2}}_*(\d M)$. If $f$ is a smooth function, then the solution $u$ is smooth as well. 
\end{theorem}

\begin{proof}
    Since $f \in H^{k+ \frac{1}{2}}_*(\d M)$, we can find a unique solution $\vp \in H^{k+2}(M)$ to the Neumann problem 
    \begin{equation}\label{neumann problem for harmonic part}
    \begin{cases}
        \Delta \vp = 0, \\
        \n \hk d \vp|_{\d M} = f,
    \end{cases}
    \end{equation}
    by \cite[\S 5, Proposition 7.7]{Taylor2011pde1}. Since $d \vp$ is an $H^{k+1}$ divergence-free $1$-form, by Theorem \ref{Theorem on existence of Beltrami fields source} we can solve the boundary value problem \eqref{Beltrami source problem} with $w = d \vp$ to obtain a unique $H^{k+2}$ solution $v$, which has zero harmonic part tangent to the boundary. Setting $u = d \vp - v$ then yields the unique $H^{k+1}$ solution to \eqref{Beltrami boundary problem} with harmonic part equal to $d \vp$. By \eqref{decomposition of H into H^t and im d}, we have $\cH^t(u) = 0$.
\end{proof}

With these results in hand, we are ready to introduce the Beltrami normal-to-tangential map, which shall be our main object of interest. It encodes the Cauchy data of Beltrami fields on a manifold with boundary just like the Dirichlet-to-Neumann map does in the case of harmonic forms.

\begin{definition}
    For real $\lambda\neq0$ that is not a Beltrami singular value, the Beltrami {\em \bf normal-to-tangential map} at frequency~$\lambda$ is the map
    \begin{equation*}
        \Sigma_{\lambda} : C^\infty_*(\d M) \to \Omega^1(\d M)
    \end{equation*}
    defined for each $f \in C^\infty_*(\d M)$ as
    \begin{equation*}
        \Sigma_{\lambda}(f) := u_t,
    \end{equation*}
    where the vector field~$u\in C^\infty(M)$ is the only solution to the boundary value problem \eqref{Beltrami boundary problem},  and where $u_t$ is the tangential part of $u|_{\d M}$, defined in~\eqref{decomposition on the boundary}.
\end{definition}

It follows from Theorem \ref{Theorem on existence of Beltrami fields} and the standard trace theorems for Sobolev spaces that $\Sigma_{\lambda}$ extends to a map from $H_*^k(\d M)$ to $H^k$ $1$-forms on $\d M$.

Note that by Equation \eqref{Beltrami equation intro, 1-form}, any solution to the Beltrami field equation with nonzero~$\lambda$ is divergence-free in the sense that
\begin{equation}\label{Beltrami fields are div free}
     d^*u = \frac{d^* (* du)}\lambda =  \frac{ *d^2 u}\lambda = 0.
\end{equation}
Here $d^*$ is the formal adjoint of $d$, which acts on $k$-forms $\a$ as
\begin{equation*}
    d^* \a = (-1)^{k} * d * \a.
\end{equation*}
As is well known, if $u$ is a 1-form, $d^* u $ coincides with the divergence of the dual vector field.

As useful consequence of this is that a Beltrami field also satisfies the first-order elliptic equation
\begin{equation}\label{E.Dirac}
    (*d + d^*)u = \lambda u,
\end{equation}
so in particular Beltrami fields with smooth boundary data are smooth (and analytic, if the domain and the boundary datum are).
Acting with $(*d + d^*)$ on both sides of Equation~\eqref{E.Dirac},
we see that~$u$ also satisfies the second-order elliptic equation
\begin{equation}\label{Hodge Laplacian of u}
    \Delta_d u  = \lambda^2 u,
\end{equation}
where the differential operator $\Delta_d:= d^*d + dd^* $ is the \emph{Hodge Laplacian}\/.  

The observation that Beltrami fields also satisfy a Helmholtz equation for the Hodge Laplacian is a key idea that we shall take advantage of below. It is important to note, however, that once we impose boundary conditions on our Beltrami fields, the analysis of the Beltrami field equation is {\em not}\/ reducible to the analysis of the Hodge Laplacian.

\begin{remark}\label{remark: introducing NT map for harmonic}
    Although we focus of the case $\lambda\neq0$ for concreteness, this assumption is not essential, provided that we prescribe the divergence-free equation $d^*u = 0$ by hand. With this understanding, $\lambda = 0$ is never Beltrami singular, and we can once again solve the following boundary value problem for a harmonic field $u \in \Omega^1(M)$ with prescribed normal component $f \in C^\8_*(\d M)$,
    \begin{equation}
        \begin{cases}
            du = 0, \ \ d^*u = 0, \\
            \n \hk u|_{\d M} = f, 
        \end{cases}
    \end{equation}
    and the solution is unique if $M$ is simply connected (see \cite[Theorem 3.4.1]{schwarz2006hodge}, for example). Indeed, this problem reduces to a Neumann problem for the scalar Laplacian, and we can define a corresponding normal-to-tangential map that sends the normal derivative of the harmonic function to its gradient along the boundary. All the results we present in this paper remain valid in this case.
\end{remark}

\subsection{Main results}\label{subsec: main results}

The main results we present in this paper is that, just as in the case of the Dirichlet-to-Neumann map for the Poisson or Helmholtz equation~\cite{Lee1989,lassas2001}, the normal-to-tangential map for Beltrami fields determines all the normal derivatives of the metric at the boundary. Furthermore, in the class of simply connected, real-analytic Riemannian manifolds, the normal-to-tangential map determines the Riemannian manifold up to an isometry. More precisely, we shall prove the following:

\begin{bigtheorem}\label{main boundary determination theorem, intro}
    If $\lambda$ is not a Beltrami singular value, the normal-to-tangential map at frequency~$\lambda$, $\Sigma_{\lambda} : C^\8(\d M) \to \Omega^1(\d M)$, is a pseudodifferential operator of order $0$ with principal symbol
    \begin{equation*}
        \sigma_0(\x) := i\frac{\x}{|\x|}.
    \end{equation*}
    Its total symbol determines all normal derivatives of the metric at the boundary.
\end{bigtheorem}

\begin{bigtheorem}\label{main reconstruction theorem}
    Suppose $(M_1, g_1)$ and $(M_2,g_2)$ are simply connected real-analytic Riemannian $3$-manifolds with boundary whose boundaries are identified, and that their Beltrami normal-to-tangential maps are equal. Then $(M_1,g_1)$ and $(M_2,g_2)$ are isometric. 
\end{bigtheorem}

The proof of Theorem \ref{main boundary determination theorem, intro} is the subject of Section \ref{sec: boundary determination}. There, we shall take advantage of the fact that every Beltrami field also satisfies the second order elliptic equation \eqref{Hodge Laplacian of u} to derive a relation between the Beltrami normal-to-tangential map and the Dirichlet-to-Neumann map for the Hodge Laplacian, whose full symbol is easily computed using the methods of \cite{Lee1989}. Deriving this connection is not immediate because, as discussed in the preceding subsection, the boundary problem for the Beltrami field equation does not reduce to a boundary problem for the corresponding second order operator, so further elaboration is necessary. We then use this relation and knowledge of the symbol of the Dirichlet-to-Neumann map to compute the symbol of the Beltrami normal-to-tangential map, and show that it determines the Taylor series of the metric on the boundary.

The proof of Theorem \ref{main reconstruction theorem} is presented in
Section \ref{sec: recovering manifold}. There, we shall adapt the
approach of \cite{lassas2003} and \cite{Krupchyk2011}, based on the
real-analyticity of the Green's functions, in which the manifold is
recovered by embedding it in a suitable Sobolev space. However, as the
source term in \eqref{Beltrami source problem} is restricted to
divergence-free fields, traditional Green's functions arising from
delta functions are not available to us. As a key idea of the proof,
we shall therefore introduce in Section \ref{subsec: analogue of
  Greens functions} a class of functions, which we call the
$b$-fields, that will play a role in the analysis of the Beltrami
field equation analogous to that of Green's functions
in~\cite{lassas2001}. These $b$-fields have an interesting physical
interpretation: they describe the magnetic field created by a small current loop. Accordingly, they are parametrized by the location and orientation of the loop (or vortex), which one associated with a point in the unit sphere bundle of the manifold. After establishing certain useful properties of $b$-fields, we proceed with the proof of Theorem \ref{main reconstruction theorem} in Section \ref{subsec: embedding M into H^s}, adapting the proofs found in \cite{lassas2003, Krupchyk2011} and emphasizing the interesting changes made necessary by the introduction of the $b$-fields. 

To readily see why having to deal with $b$-fields instead of Green's functions introduces essential new difficulties, it is useful to consider the differences between Theorem \ref{main reconstruction theorem} and the kinds of uniqueness results one usually obtains for the Dirichlet-to-Neumann map for Laplacians with analytic data, such as those in \cite{lassas2003, Krupchyk2011}. First, we require that the manifolds $M_1$ and $M_2$ in Theorem \ref{main reconstruction theorem} be simply connected. This is a direct consequence of the fact that in order for the solution of the boundary value problem \eqref{Beltrami boundary problem} to be unique, we require the harmonic part to vanish. We discuss this point in greater detail in Remark \ref{remark on simply connectedness}, following the proof of Proposition \ref{proposition: b-fields agree near boundary}, where the role that simple-connectedness plays in the proof of Theorem \ref{main reconstruction theorem} becomes manifest. 

Second, it would be reasonable to ask why we have not stated Theorem \ref{main reconstruction theorem} in the case of local data, that is, when $\Sigma_{\lambda}$ is known only on an open subset $\Gamma$ of the boundary. Indeed, the arguments used in \cite{Krupchyk2011} to reconstruct the manifold from knowledge of the Green's form for the Hodge Laplacian apply equally well in the case of local data. However, as remarked above, since there does not exist a Green's form for the Beltrami field equation, we use instead an appropriate analogue, which we call the $b$-fields. These $b$-fields have two properties that pose a difficulty in extending Theorem \ref{main reconstruction theorem} to the case of local data. First, these $b$-fields are not generated by point sources, but by loops of small but non-zero radius; and second, unlike the Green's form, which is symmetric in its two arguments, the $b$-fields shall depend on a parameter that is completely distinct from its argument, and so there is no symmetry to exploit. We discuss this in greater detail in Remark \ref{remark: no local theorem}, after it becomes clear how these properties enter into the proof of Theorem \ref{main reconstruction theorem}.

% refer to remark on simply connectedness, and...

\section{Boundary determination for the Beltrami normal-to-tangential map}\label{sec: boundary determination}

The goal of this section is to show that the Beltrami normal-to-tangential map defined in Section \ref{sec: intro} is a pseudodifferential operator of order $0$ whose total symbol in a boundary chart determines the Taylor series of the metric at the boundary. The strategy will be to use the relation between the Beltrami operator and the Hodge Laplacian for $1$-forms to relate the Beltrami normal-to-tangential map to the Hodge Dirichlet-to-Neumann map, whose symbol is easily computed using methods similar to those in \cite{Lee1989}. We then argue that the symbol of the Beltrami normal-to-tangential map inductively determines the Taylor series of the metric at the boundary. The fundamental results regarding pseudodifferential operators and their symbols that we shall employ in this section may be found in any good text on the subject; see \cite{taylor1981} for example. 

In Section \ref{subsection: relation between DN and NT}, we derive an expression in boundary normal coordinates that relates the Beltrami normal-to-tangential map to the Dirichlet-to-Neumann map of the Hodge Laplacian for $1$-forms. In Section \ref{subsection: computation of total symbol of hodge DN}, we explicitly compute the total symbol of the Dirichlet-to-Neumann map for $1$-forms. Finally, in Section \ref{subsection: boundary determination for NT map} we use the relation derived in Section \ref{subsection: relation between DN and NT} and the results of Section \ref{subsection: computation of total symbol of hodge DN} to show that the total symbol of the Beltrami normal-to-tangential map inductively determines all normal derivatives of the metric at the boundary.

\subsection{Deriving a relation between the Hodge DN map and the Beltrami NT map} \label{subsection: relation between DN and NT}

In this Section, we want to derive a relation between the Dirichlet-to-Neumann map for differential $1$-forms and the normal-to-tangential map for Beltrami fields. As mentioned above, there are several definitions of the Dirichlet-to-Neumann map for $k$-forms found in the literature; see for example \cite{Joshi2005, belishev2008, sharafutdinov2013}. The one that shall be most convenient for our purposes is the first definition provided in \cite{Joshi2005}:

\begin{definition}\label{Hodge DN map definition}
Let $M$ be an $n$-dimensional compact manifold with boundary $\d M$, and let $g$ be a Riemannian metric on $M$. Let $\Lambda^k T^*M$ be the bundle of $k$-forms on $M$. For any $\lambda^2 \notin \Spec{\left( \Delta_d \right)}$, we define the {\bf Hodge Dirichlet-to-Neumann map}
\begin{equation*}
    \Lambda_{g, \lambda^2} : C^\8\left( \Lambda^k T^*M |_{\d M} \right) \to C^\8\left( \Lambda^k T^*M |_{\d M} \right)
\end{equation*}
as follows. For $v \in C^\8\left( \Lambda^k T^*M|_{\d M} \right)$, we may solve the Dirichlet problem
\begin{equation}\label{Dirichlet problem for chi}
    \begin{cases}
    \Delta_d u - \lambda^2 u = 0, \\
    u|_{\d M} = v,
    \end{cases}
\end{equation}
to obtain a unique solution $u \in C^\8(\Lambda^k T^*M)$. We then define $\Lambda_{g,\lambda^2}(v) := \d_n u |_{\d M}$ where in boundary normal coordinates $(x^i)$, $\d_n u |_{\d M}$ is defined as
\begin{equation}\label{Hodge DN map in boundary normal coords}
    \d_n u |_{\d M} := \sum_{I} \d_n u_I |_{\d M} \, dx^{I}.
\end{equation}
\end{definition}

\begin{remark}
It follows from the definition of boundary normal coordinates that the right-hand-side of \eqref{Hodge DN map in boundary normal coords} is simply the Lie derivative of $u$ with respect the inward unit normal, restricted to $\d M$.
\end{remark}

We now want to prove the following relationship between the Hodge Dirichlet-to-Neumann map for $1$-forms as defined above, and the Beltrami normal-to-tangential map:

\begin{proposition}\label{prop: relation between DN and NT map}
Let $\Sigma_{\lambda}$ be the normal-to-tangential map for Beltrami fields corresponding to $\lambda \in \R \setminus \{ 0 \}$, such that $\lambda$ is not a Beltrami singular value, and $\lambda^2$ is not in the Dirichlet spectrum of the Hodge Laplacian $\Delta_d$. Let $\Lambda_{\lambda^2}$ denote the Dirichlet-to-Neumann map of the Hodge Laplacian for $1$-forms at frequency $\lambda^2$. Then in boundary normal coordinates, for all $f \in C^{\infty}_{*}(\d M)$, we have
\begin{equation} \label{theorem eq: Sigma tangential equation}
    \Lambda_{\lambda^2} \left(\Sigma_{\lambda}(f) + f dx^n \right)_{\g} = \frac{\lambda}{\sqrt{|g|}} g_{\a \g} \epsilon^{\a \b} \Sigma_{\lambda}(f)_{\b} + \d_{\g} f 
\end{equation}
\begin{equation}
    \Lambda_{\lambda^2} \left( \Sigma_{\lambda}(f) + f dx^n \right)_3 = A^{\b} \Sigma_{\lambda}(f)_{\b} + C(\d)^{\b} \Sigma_{\lambda}(f)_{\b} +  \frac{1}{\lambda} \d_n \left( \frac{1}{\sqrt{|g|}} \right) \epsilon^{\a \g}  \d_{\a} \Sigma_{\lambda}(f)_{\g} \label{theorem eq: Sigma normal equation}
\end{equation}
where 
\begin{equation*}
    A^{\b} := \frac{1}{\sqrt{|g|}} \epsilon^{\delta \g} \epsilon^{\a \b} \d_{\delta} \left( \frac{1}{\sqrt{|g|}} g_{\a \g} \right),
\end{equation*}
and 
\begin{equation*}
    C(\d)^{\b} := \frac{1}{|g|} \epsilon^{\delta \g} g_{\a \g} \epsilon^{\a \b} \d_{\delta}.
\end{equation*} 
\end{proposition}

\begin{proof}
For $f \in C_*^\infty(\d M)$, let $u$ be the Beltrami field whose normal component on the boundary is given by $f$. The Hodge star on the basis of $2$-forms in boundary normal coordinates is given by
\begin{equation*}
    * \left( dx^1 \wedge dx^2 \right) = \sqrt{|g|} \left( g^{11}g^{22} - g^{12}g^{21} \right) dx^3 = \frac{1}{\sqrt{|g|}} dx^3,
\end{equation*}
\begin{equation*}
    * \left( dx^2 \wedge dx^3 \right) = \sqrt{|g|} \left( g^{22} dx^1 - g^{21} dx^2 \right),
\end{equation*}
\begin{equation*}
    * \left( dx^3 \wedge dx^1 \right) = \sqrt{|g|} \left( g^{11} dx^2 - g^{12} dx^1 \right). 
\end{equation*}
Applying this to 
\begin{equation*}
    du = (\d_1 u_2 - \d_2 u_1) dx^1\wedge dx^2 + (\d_2 u_3 - \d_3 u_2) dx^2 \wedge dx^3 + (\d_3 u_1 - \d_1 u_3) dx^3 \wedge dx^1
\end{equation*}
we can express $* du$ as
\begin{align}
    *du &= \sqrt{|g|} (\d_2 u_3 - \d_3 u_2)(g^{22} dx^1 - g^{21} dx^2) + \sqrt{|g|} (\d_3 u_1 - \d_1 u_3)(g^{11} dx^2 - g^{12} dx^1) \nonumber \\
    & \ \ \ \ \ \ \ + \frac{1}{\sqrt{|g|}} (\d_1 u_2 - \d_2 u_1) dx^3
\end{align}
Thus, in boundary normal coordinates, the Beltrami field equations take the form
\begin{equation}\label{Beltrami in normal coordinates u_1}
    \sqrt{|g|} g^{22} (\d_2 u_3 - \d_3 u_2) - \sqrt{|g|} g^{12} (\d_3 u_1 - \d_1 u_3) = \lambda u_1
\end{equation}
\begin{equation}\label{Beltrami in normal coordinates u_2}
    \sqrt{|g|} g^{11} (\d_3 u_1 - \d_1 u_3) - \sqrt{|g|} g^{21} (\d_2 u_3 - \d_3 u_2) = \lambda u_2
\end{equation}
\begin{equation}\label{Beltrami in normal coordinates u_3}
    \frac{1}{\sqrt{|g|}} (\d_1 u_2 - \d_2 u_1) = \lambda u_3.
\end{equation}
Rearranging Equations \eqref{Beltrami in normal coordinates u_1}--\eqref{Beltrami in normal coordinates u_2} yields
\begin{equation}\label{Beltrami normal coord rearrange 1}
    \sqrt{|g|} \left( g^{11} \d_3 u_1 + g^{21} \d_3 u_2 \right) = \lambda u_2 +  \sqrt{|g|} \left( g^{11} \d_1 + g^{21} \d_2 \right) u_3,
\end{equation}
\begin{equation}\label{Beltrami normal coord rearrange 2}
    \sqrt{|g|} \left( g^{12} \d_3 u_1 + g^{22} \d_3 u_2 \right) = -\lambda u_1 + \sqrt{|g|} \left( g^{12} \d_1 + g^{22} \d_2 \right) u_3.
\end{equation}
We can rewrite Equations \eqref{Beltrami normal coord rearrange 1}--\eqref{Beltrami normal coord rearrange 2} as
\begin{equation}\label{beltrami tangential index form take 1}
    \sqrt{|g|} g^{\a \b} \d_3 u_{\b} = \lambda \epsilon^{\a \b} u_{\b} + \sqrt{|g|}g^{\a \b} \d_{\b} u_3,
\end{equation}
where $\epsilon^{\a \b}$ is the usual permutation symbol on $2$ indices. Dividing Equation \eqref{beltrami tangential index form take 1} by $\sqrt{|g|}$ and contracting with $g_{\a \g}$, we can rewrite it as
\begin{equation}\label{beltrami tangential index form take 2}
    \d_3 u_{\g} = \frac{\lambda}{\sqrt{|g|}} g_{\a \g} \epsilon^{\a \b} u_{\b} + \d_{\g} u_3
\end{equation}
It remains to compute the derivative of $u_3$ with respect to $x^3$. To this end, we differentiate Equation \eqref{Beltrami in normal coordinates u_3} with respect to $x^3$ to obtain
\begin{align}
    \d_3 u_3 &= \frac{1}{\lambda \sqrt{|g|}} (\d_1 \d_3 u_2 - \d_2 \d_3 u_1) + \frac{1}{\lambda} \d_3 \left( \frac{1}{\sqrt{|g|}} \right)(\d_1 u_2 - \d_2 u_1) \nonumber \\
    &= \frac{1}{\lambda \sqrt{|g|}} \epsilon^{\delta \g} \d_{\delta} \d_3 u_{\g} + \frac{1}{\lambda} \d_3 \left( \frac{1}{\sqrt{|g|}} \right) \epsilon^{\a \g}  \d_{\a} u_{\g} \label{Beltrami normal index form take 1}.
\end{align}
By substituting Equation \eqref{beltrami tangential index form take 2} into Equation \eqref{Beltrami normal index form take 1}, we obtain
\begin{align}
    \d_3 u_3 &= \frac{1}{\lambda \sqrt{|g|}} \epsilon^{\delta \g} \d_{\delta} \d_3 u_{\g} + \frac{1}{\lambda} \d_3 \left( \frac{1}{\sqrt{|g|}} \right) \epsilon^{\a \g}  \d_{\a} u_{\g} \nonumber \\
    &= \frac{1}{\lambda \sqrt{|g|}} \epsilon^{\delta \g} \d_{\delta} \left(  \frac{\lambda}{\sqrt{|g|}} g_{\a \g} \epsilon^{\a \b} u_{\b} + \d_{\g} u_3 \right) + \frac{1}{\lambda} \d_3 \left( \frac{1}{\sqrt{|g|}} \right) \epsilon^{\a \g}  \d_{\a} u_{\g} \nonumber \\
    &= \frac{1}{\lambda \sqrt{|g|}} \epsilon^{\delta \g} \d_{\delta} \left(  \frac{\lambda}{\sqrt{|g|}} g_{\a \g} \epsilon^{\a \b} u_{\b} \right) + \frac{1}{\lambda} \d_3 \left( \frac{1}{\sqrt{|g|}} \right) \epsilon^{\a \g}  \d_{\a} u_{\g}, \label{Beltrami normal index form take 2}
\end{align}
since the term $\epsilon^{\delta \g} \d_{\delta} \d_{\g} u_3$ vanishes. We thus obtain
\begin{equation}\label{Beltrami normal index form take 3}
    \d_3 u_3 = \frac{1}{\sqrt{|g|}} \epsilon^{\delta \g} \epsilon^{\a \b} \d_{\delta} \left( \frac{1}{\sqrt{|g|}} g_{\a \g} \right) u_{\b} + \frac{1}{|g|} \epsilon^{\delta \g} g_{\a \g} \epsilon^{\a \b} \d_{\delta} u_{\b} +  \frac{1}{\lambda} \d_3 \left( \frac{1}{\sqrt{|g|}} \right) \epsilon^{\a \g}  \d_{\a} u_{\g}.
\end{equation}
Now, since $u_3|_{\d M} = f$, we have $u_{\g}|_{\d M} = \Sigma_{\lambda}(f)_{\g}$ by definition of the normal-to-tangential map. On the other hand, $u$ is also the unique $1$-form on $M$ satisfying
\begin{equation*}
    \begin{cases}
    \Delta_d u = \lambda^2 u, \\
    u|_{\d M} = \Sigma_{\lambda}(f) + f dx^n.
    \end{cases}
\end{equation*}
Therefore, combining Equations \eqref{beltrami tangential index form take 2} and \eqref{Beltrami normal index form take 3}, we have
\begin{align} 
    \Lambda_{\lambda^2} \left(\Sigma_{\lambda}(f) + f dx^3\right)_{\g} &= \d_3 u_{\g} |_{\d M} \nonumber \\
    &= \frac{\lambda}{\sqrt{|g|}} g_{\a \g} \epsilon^{\a \b} \Sigma_{\lambda}(f)_{\b} + \d_{\g} f,  \label{Sigma tangential equation}
\end{align}
\begin{align}
    \Lambda_{\lambda^2} \left( \Sigma_{\lambda}(f) + f dx^3 \right)_3 &= \d_3 u_3 |_{\d M} \nonumber \\
    &= A^{\b} \Sigma_{\lambda}(f)_{\b} + C(\d)^{\b} \Sigma_{\lambda}(f)_{\b} +  \frac{1}{\lambda} \d_3 \left( \frac{1}{\sqrt{|g|}} \right) \epsilon^{\a \g}  \d_{\a} \Sigma_{\lambda}(f)_{\g}, \label{Sigma normal equation}
\end{align}
where 
\begin{equation*}
    A^{\b} := \frac{1}{\sqrt{|g|}} \epsilon^{\delta \g} \epsilon^{\a \b} \d_{\delta} \left( \frac{1}{\sqrt{|g|}} g_{\a \g} \right),
\end{equation*}
and 
\begin{equation*}
    C(\d)^{\b} := \frac{1}{|g|} \epsilon^{\delta \g} g_{\a \g} \epsilon^{\a \b} \d_{\delta},
\end{equation*}
as defined in the statement of Proposition \ref{prop: relation between DN and NT map}. This complete the proof.
\end{proof}
Equations \eqref{Sigma tangential equation} and \eqref{Sigma normal equation} constitute key relation between $\Lambda_{\lambda^2}$ and $\Sigma_{\lambda}$ that we wish to exploit. We may also write Equations \eqref{Sigma tangential equation}--\eqref{Sigma normal equation} in matrix form,
\begin{equation}\label{Sigma matrix equation}
    \begin{pmatrix} \Lambda_{\lambda^2}^{tt} & \Lambda_{\lambda^2}^{tn} \\ \Lambda_{\lambda^2}^{nt} & \Lambda_{\lambda^2}^{nn} \end{pmatrix} \begin{pmatrix} \Sigma_{\lambda}(f) \\ f \end{pmatrix} = \begin{pmatrix} \lambda |g|^{-\frac{1}{2}} G J \Sigma_{\lambda}(f) + d_{\d M} f  \\ A^{\b} \Sigma_{\lambda}(f)_{\b} + C(\d)^{\b} \Sigma_{\lambda}(f)_{\b} +  \lambda^{-1} \d_3 \left( |g|^{-\frac{1}{2}} \right) \epsilon^{\a \g}  \d_{\a} \Sigma_{\lambda}(f)_{\g} \end{pmatrix},
\end{equation}
where the blocks correspond to the splitting of $T^*M|_{\d M}$ into tangential and normal components, the $2 \times 2$ matrices $G$ and $J$ are given by $G_{\a \b} := g_{\a \b}$ and $J^{\a \b} := \epsilon^{\a \b}$, and $d_{\d M}$ is the exterior derivative on $\d M$. We note in passing that the first term on the right-hand-side of Equation \eqref{Sigma tangential equation} is easily seen to equal $\lambda *_{\d M} \Sigma(f)$, and therefore the tangential equation \eqref{Sigma tangential equation} can be written as
\begin{equation}\label{Sigma tangential equation global}
    \Lambda_{\lambda^2}^{tt} \Sigma(f) + \Lambda_{\lambda^2}^{tn}(f) = \lambda *_{\d M} \Sigma(f) + d_{\d M} f.
\end{equation}

\begin{corollary}
The Beltrami normal-to-tangential map $\Sigma_{\lambda}$ is a classical pseudodifferential operator of order $0$ with principal symbol
\begin{equation}\label{principal symbol of Sigma}
    \sigma_0(x,\x) = i\frac{\x}{|\x|_g}.
\end{equation}
\end{corollary}

\begin{proof}
    The tangential equation \eqref{Sigma tangential equation global} yields
    \begin{equation}\label{before applying parametrix}
        \left(\Lambda_{\lambda^2}^{tt} - \lambda *_{\d M} \right) \Sigma_{\lambda}(f) = d_{\d M} f - \Lambda_{\lambda^2}^{tn}(f).
    \end{equation}
    Since $\Lambda^{tt}_{\lambda^2}$ is an elliptic classical pseudodifferential operator of order $1$, so is $\Lambda^{tt}_{\lambda^2} - \lambda *_{\d M}$, and so it has a parametrix $P$, which is an elliptic classical pseudodifferential operator of order $-1$ satisfying
    \begin{equation*}
        P \left(\Lambda_{\lambda^2}^{tt} - \lambda *_{\d M} \right) = \id + \sS,
    \end{equation*}
    where $\sS$ is a smoothing operator. Applying $P$ to Equation \eqref{before applying parametrix} yields
    \begin{equation*}
        \Sigma_{\lambda}(f) + \sS \Sigma_{\lambda}(f) = P \left( d_{\d M} f - \Lambda^{tn}_{\lambda^2}(f) \right).
    \end{equation*}
    By the the Sobolev estimates given by Theorem \ref{Theorem on existence of Beltrami fields}, $\sS \circ \Sigma_{\lambda}$ extends to a map of distributions into smooth $1$-forms, and so it is a smoothing operator. Therefore, we have that $\Sigma_{\lambda}$ is a classical pseudodifferential operator of order $0$. Recalling that the principal symbol of $\Lambda^{tt}_{\lambda^2}$ is $|\x|_g$, we take the degree-$1$ part of Equation \eqref{Sigma tangential equation global} to obtain
    \begin{equation*}
        |\x|_g \sigma_0(x,\x) = i \x,
    \end{equation*}
    which yields Equation \eqref{principal symbol of Sigma}.
\end{proof}

\subsection{Computing the symbol of Hodge DN map}\label{subsection: computation of total symbol of hodge DN}

The relation \eqref{Sigma matrix equation} between $\Sigma_{\lambda}$ and $\Lambda_{\lambda^2}$ shall be the key ingredient used in Section \ref{subsection: boundary determination for NT map} to prove that the full symbol of $\Sigma_{\lambda}$ determines the normal derivatives of the metric on $\d M$. Since that proof requires an explicit knowledge of the total symbol of the Dirichlet-to-Neumann map for $1$-forms, and since it does not appear that these computations have been done in the literature, we take some time in this Section to apply the standard recipe of \cite{Lee1989} to find the explicit form of the total symbol of $\Lambda_{\lambda^2}$. For the remainder of this Section and the following one, we shall fix a frequency $\lambda$ and omit the subscripts on $\Lambda_{\lambda^2}$ and $\Sigma_{\lambda}$.

We employ the same method of Lee and Uhlmann as in \cite{Lee1989} to the Hodge Laplacian to get the total symbol of the Dirichlet-to-Neumann map in asymptotic series form. So we seek a factorization of the form (where $D = - i \d$)
\begin{equation*}
    \Delta_d - \lambda^2 = (D_n + i E - i B)(D_n + i B)
\end{equation*}
for some pseudodifferential operator $B$. On the other hand, for the Hodge Laplacian $\Delta_d$ on $1$-forms, we have by the Weitzenb\"ock formula,
\begin{equation}\label{weitzenbock}
    \Delta_d u = \nabla^* \nabla u + u \circ \Ric
\end{equation}
where we view $\Ric$ as an endomorphism of $TM$. We choose boundary normal coordinates $(x^1, x^2, x^n)$, and as usual, we let Latin indices run over all $3$ coordinates, whereas Greek indices shall run over only the boundary coordinates. We now want to express Equation \eqref{weitzenbock} in these coordinates:
\begin{equation*}
    \left( \Delta_d u \right)_{\l} = - g^{ij} \left( \nabla_i \nabla_j u \right)_{\l} + g^{ij} \Gamma^k_{ij} \left( \nabla_k u \right)_{\l} + \Ric^p_{\l} u_p
\end{equation*}
Note that we have
\begin{equation*}
    \left( \nabla_k u \right)_{\l} = \d_k u_{\l} - \Gamma^p_{k \l} u_p
\end{equation*}
And therefore
\begin{equation*}
    \left( \nabla_i \nabla_j u \right)_{\l} = \d_i \d_j u_{\l} - \Gamma^p_{j \l} \d_i u_p - \Gamma^p_{i \l} \d_j u_p - \left( \d_i \Gamma^p_{j \l} \right) u_p + \Gamma^q_{i \l} \Gamma^p_{j q} u_p
\end{equation*}
and therefore we get
\begin{align}
    \left( \Delta_d u \right)_{\l} &= -g^{ij}\d_i \d_j u_{\l} + 2 g^{ij}\Gamma^p_{i \l} \d_j u_p + g^{ij} \left( \d_i \Gamma^p_{j \l} \right) u_p - g^{ij} \Gamma^q_{i \l} \Gamma^p_{j q} u_p  + g^{ij} \Gamma^k_{ij} \d_k u_{\l} - g^{ij} \Gamma^k_{ij} \Gamma^p_{k \l} u_p \nonumber \\ &\ \ \ \ \ + \Ric^p_{\ \l} u_p
\end{align}
The Ricci tensor is given by
\begin{equation}\label{Ric}
    \Ric^i_{\ \l} = g^{ij}\Ric_{j \l} = g^{ij}\d_k \Gamma^k_{\l j} - g^{ij}\d_{\l} \Gamma^k_{k j} + g^{ij}\Gamma^k_{k p} \Gamma^p_{\l j} - g^{ij}\Gamma^k_{\l p} \Gamma^p_{kj}.
\end{equation}
We now separate out the normal derivatives and tangential derivatives:
\begin{align}
    \left( \Delta_d u \right)_{\l} - \lambda^2 u_{\l} &= -\d_n^2 u_{\l} -g^{\a \b}\d_\a \d_\b u_{\l} + 2 \Gamma^{p}_{n \l} \d_n u_p + 2 g^{\a \b}\Gamma^p_{\a \l} \d_\b u_p + \left( \d_n \Gamma^p_{n \l} \right) u_p + g^{\a \b} \left( \d_{\a} \Gamma^p_{\b \l} \right) u_p \nonumber \\
    & \ \ \ \ - g^{ij} \Gamma^q_{i \l} \Gamma^p_{j q} u_p  + g^{ij} \Gamma^n_{ij} \d_n u_{\l} + g^{ij} \Gamma^{\a}_{ij} \d_\a u_{\l} - g^{ij} \Gamma^k_{ij} \Gamma^p_{k \l} u_p + \Ric^p_{\ \l} u_p \nonumber \\
    &= D_n^2 u_\l + i \left( 2 \Gamma^p_{n \l} + g^{\a \b} \Gamma^n_{\a \b}  \right) D_n u_{\l} + Q_2 u + Q_1 u + Q_0 u
\end{align}
where for $i \in \{0,1,2\}$, $Q_i$ is the $i$-th order differential operator on $\d M$ defined by
\begin{align}
    \left(Q_2 u\right)_{\l} &:= -g^{\a \b}\d_\a \d_\b u_{\l} \\
    \left(Q_1 u \right)_{\l} &:= \left( 2 g^{\a \b} \Gamma^p_{\a \l} \d_\b + \delta^p_{\l} g^{\g \lambda} \Gamma^{\a}_{\g \lambda} \d_\a \right) u_p \\
    \left( Q_0 u \right)_{\l} &:= \left( \d_n \Gamma^p_{n \l} + g^{\a \b} \d_\a \Gamma^p_{\b \l} - g^{ij} \Gamma^q_{i \l} \Gamma^p_{j q} - g^{ij} \Gamma^k_{ij} \Gamma^p_{k \l} + \Ric^p_{\ \l} - \lambda^2 \delta^p_{\l} \right) u_p. \label{Q_0}
\end{align}
So, letting $(Eu)_{\l} = 2 \Gamma^p_{n \l} u_p + g^{\a \b} \Gamma^n_{\a \b} u_\l$, we have
\begin{equation*}
    \Delta u = D_n^2 u + i E D_n u + Q_2 u + Q_1 u + Q_0 u,  
\end{equation*}
which upon comparing with the factorization, gives us
\begin{equation*}
    i [ D_n , B ] - EB + B^2 = Q_2 + Q_1 + Q_0.
\end{equation*}
Taking total symbols of this equation, we get
\begin{equation}\label{equation for total symbol of b}
    \d_n b - E b + \sum \frac{1}{\a!} \d_{\x}^\a b D_{x'}^\a b = q_2 + q_1 + q_0.
\end{equation}
Writing 
\begin{equation*}
    b(x,\x) = \sum_{m \leq 1} b_m(x,\x)
\end{equation*}
where each $b_m$ is positive-homogeneous in $\x$ of degree $m$, the degree-$2$ part of Equation \eqref{equation for total symbol of b} gives $b_1 = \sqrt{q_2} = |\x|$ as expected. The degree $1$ equation yields
\begin{equation}\label{degree 1 equation hodge symbol}
    \d_n b_1 - E b_1 + 2 b_0 b_1 + \sum \d_\x b_1 \cdot D_{x'} b_1 = q_1.
\end{equation}
Now, before proceeding, we remark that for the proof of Theorem \ref{Main Theorem for Beltrami Calderon} in Section \ref{subsection: boundary determination for NT map}, we shall only need to know the explicit dependence of $b_0$ on $\d_n g^{\a \b}|_{\d M}$ modulo terms involving $g_{\a \b}|_{\d M}$ and tangential derivatives thereof. In fact, for $m \leq 0$, we shall only need the explicit dependence of $b_{m}$ on $\d^{|m|+1}_n g^{\a \b}|_{\d M}$ modulo terms involving lower-order normal derivatives of $g$ and their tangential derivatives on $\d M$. For convenience, therefore, we introduce the following notation. Let $\cO_m$ denote any quantity on $\d M$ which depends on normal derivatives of $g$ up to order $|m|$, and tangential derivatives thereof. Thus, letting $\omega = |\x|^{-1} \x$ in Equation \eqref{degree 1 equation hodge symbol}, we have
\begin{align}
    b_0 &= \frac{1}{2 b_1} \left( q_1 + E b_1 - \d_n b_1 \right) + \cO_0 \nonumber \\
    &= \frac{1}{2 |\x|} \left( 2 i g^{\a \b} \Gamma^p_{\a \l} \x_\b + i \delta^p_{\l} g^{\g \lambda} \Gamma^{\a}_{\g \lambda} \x_\a + 2 |\x| \Gamma^p_{n \l} + |\x| \delta^p_{\l} g^{\a \b} \Gamma^n_{\a \b} - \frac{1}{2 |\x|} \left( \d_n g^{\a \b} \right) \x_\a \x_\b \right) + \cO_0\nonumber \\
    &= i \omega_\a \left( g^{\a \b} \Gamma^p_{\b \l} + \frac{1}{2} \delta^p_{\l} g^{\g \lambda} \Gamma^{\a}_{\g \lambda}  \right) + \omega_\a \omega_\b \left( g^{\a \b} \Gamma^p_{n \l} + \frac{1}{2} g^{\a \b} \delta^p_{\l} g^{\g \lambda} \Gamma^n_{\g \lambda} - \frac{1}{4} \delta^p_{\l} \d_ng^{\a \b} \right) + \cO_0 \label{b_0 full}
\end{align}
For later use, we decompose $b_0$ into its tangential and normal components, modulo $\cO_0$:
\begin{equation}\label{b_0^nn}
    b_0^{nn} = \omega_\a \omega_\b \left( \frac{1}{2} g^{\a \b} g^{\g \lambda} \Gamma^n_{\g \lambda} - \frac{1}{4} \d_ng^{\a \b} \right) + \cO_0
\end{equation}
\begin{equation}\label{b_0^tn}
    \left( b_0^{tn} \right)_\n = i \omega_\a g^{\a \b} \Gamma^n_{\b \n} + \cO_0
\end{equation}
\begin{equation}\label{b_0^nt}
    \left( b_0^{nt} \right)^\m = i \omega_\a g^{\a \b} \Gamma^{\m}_{n \b} + \cO_0
\end{equation}
\begin{equation}\label{b_0^tt}
    \left( b_0^{tt} \right)^{\m}_{\n} =  \omega_\a \omega_\b \left( g^{\a \b} \Gamma^{\m}_{n \n} + \frac{1}{2} g^{\a \b} \delta^{\m}_{\n} g^{\g \lambda} \Gamma^n_{\g \lambda} - \frac{1}{4} \delta^{\m}_{\n} \d_ng^{\a \b} \right) + \cO_0.
\end{equation}
We also want to derive expressions for $b_{-1}$ modulo terms in $\cO_1$. For this, note that the degree $0$ part of Equation \eqref{equation for total symbol of b} gives us
\begin{equation*}
    b_{-1} = \frac{1}{2 |\x|} \left( q_0 - \d_n b_0 \right) + \cO_{1}.
\end{equation*}
Here we introduce yet another notation that shall be very convenient for us in Section \ref{subsection: boundary determination for NT map}. Let $\cD$ denote any one of the following equivalent quantities,
\begin{align}\label{defining cD}
    \cD &:= g^{\a \b} \Gamma^n_{\a \b} = - \Gamma^{\a}_{n \a} = - \frac{1}{2}g^{\a \b} \d_n g_{\a \b} = \frac{1}{2} g_{\a \b} \d_n g^{\a \b} = -\frac{1}{\sqrt{|g|}} \d_n \left( \sqrt{|g|} \right),
\end{align}
where $|g|$ denotes $\det{g}$ in these coordinates. Then using Equation \eqref{b_0 full}, we obtain
\begin{equation*}
    \d_n \left( b_0 \right)^p_{\ \l} = i \omega_{\a} g^{\a \b} \d_n \Gamma^p_{\b \l} + \omega_{\a} \omega_{\b} \left( g^{\a \b} \d_n \Gamma^p_{n \l} + \frac{1}{2}  g^{\a \b}\delta^p_{\l} \d_n \cD - \frac{1}{4} \delta^p_{\l} \d_n^2 g^{\a \b} \right) + \cO_{1}.
\end{equation*}
Using Equations \eqref{Q_0} and \eqref{Ric}, we obtain
\begin{align}
    \left( q_0 \right)^p_{\ \l} &= \d_n \Gamma^p_{n \l} + \Ric^p_{\ \l} + \cO_1 \nonumber \\
    &= \d_n \Gamma^p_{n \l} + g^{p j} \d_n \Gamma^n_{\l j} - g^{p j} \d_{\l} \Gamma^{\a}_{j \a} + \cO_1 \label{q_0 + O_1}.
\end{align}
Therefore, we have
\begin{align}
    \left( b_{-1} \right)^p_{\ \l} &= i \omega_{\a} g^{\a \b} \d_n \Gamma^p_{\b \l} + \frac{1}{2 |\x|} \omega_{\a} \omega_{\b} \bigg( g^{\a \b} \d_n \Gamma^p_{n \l} + g^{\a \b}g^{p j} \d_n \Gamma^n_{\l j} - g^{\a \b} g^{p j} \d_{\l} \Gamma^{\a}_{j \a} + g^{\a \b} \d_n \Gamma^p_{n \l}  \nonumber \\
    &\ \ \ \ \ \ \ + \frac{1}{2}  g^{\a \b}\delta^p_{\l} \d_n \cD -\frac{1}{4} \delta^p_{\l} \d_n^2 g^{\a \b} \bigg) + \cO_1.
\end{align}
As with $b_0$, we separate out the tangential and normal blocks of $b_{-1}$,
\begin{equation}\label{b_-1 tt}
    \left( b^{tt}_{-1} \right)^{\m}_{\ \n} = \frac{1}{2 |\x|} \omega_{\a} \omega_{\b} \left( g^{\a \b} \d_n \Gamma^{\m}_{n \n} + g^{\a \b}g^{\m \g} \d_n \Gamma^n_{\n \g} + g^{\a \b} \d_n \Gamma^{\m}_{n \n} + \frac{1}{2}  g^{\a \b}\delta^{\m}_{\n} \d_n \cD -\frac{1}{4} \delta^{\m}_{\n} \d_n^2 g^{\a \b} \right) + \cO_1,
\end{equation}
\begin{equation}\label{b_-1 tn}
    \left( b^{tn}_{-1} \right)_{\ \n} = \frac{i}{2|\x|} \omega_{\a} g^{\a \b} \d_n \Gamma^n_{\b \n} + \cO_1,
\end{equation}
\begin{equation}\label{b_-1 nt}
    \left( b^{nt}_{-1} \right)^{\m} = \frac{i}{2|\x|} \omega_{\a} g^{\a \b} \d_n \Gamma^{\m}_{\b n} + \cO_1,
\end{equation}
\begin{equation}\label{b_-1 nn}
    b_{-1}^{nn} = \frac{1}{2 |\x|} \omega_{\a} \omega_{\b} \left( -g^{\a \b} \d_{n} \Gamma^{\g}_{n \g} + \frac{1}{2} g^{\a \b} \d_n \cD - \frac{1}{4} \d_n^2 g^{\a \b} \right) + \cO_1.
\end{equation}
For $m \leq -2$, the degree $m$ part of Equation \eqref{equation for total symbol of b} yields
\begin{equation*}
    b_{m} = -\frac{1}{2 |\x|} \d_n b_{m+1} + \cO_{|m|}.
\end{equation*}
From Equations \eqref{b_-1 tt}--\eqref{b_-1 nn}, it is clear that by induction, we obtain
\begin{align}\label{b_-m tt}
    \left( b^{tt}_{m} \right)^{\m}_{\ \n} &= \frac{(-1)^{|m|+1}}{2^{|m|} |\x|^{|m|}} \omega_{\a} \omega_{\b} \bigg( g^{\a \b} \d^{|m|}_n \Gamma^{\m}_{n \n} + g^{\a \b}g^{\m \g} \d^{|m|}_n \Gamma^n_{\n \g} + g^{\a \b} \d^{|m|}_n \Gamma^{\m}_{n \n} + \frac{1}{2}  g^{\a \b}\delta^{\m}_{\n} \d^{|m|}_n \cD \nonumber \\
    &\ \ \ \ \ \ \ -\frac{1}{4} \delta^{\m}_{\n} \d_n^{|m|+1} g^{\a \b} \bigg) + \cO_{|m|},
\end{align}
\begin{equation}\label{b_-m tn}
    \left( b^{tn}_{m} \right)_{\ \n} = \frac{i(-1)^{|m|+1}}{2^{|m|} |\x|^{|m|}} \omega_{\a} g^{\a \b} \d_n^{|m|} \Gamma^n_{\b \n} + \cO_{|m|},
\end{equation}
\begin{equation}\label{b_-m nt}
    \left( b^{nt}_{m} \right)^{\m} = \frac{i(-1)^{|m|+1}}{2^{|m|} |\x|^{|m|}} \omega_{\a} g^{\a \b} \d_n^{|m|} \Gamma^{\m}_{\b n} + \cO_{|m|},
\end{equation}
\begin{equation}\label{b_-m nn}
    b_{m}^{nn} = \frac{(-1)^{|m|+1}}{2^{|m|} |\x|^{|m|}} \omega_{\a} \omega_{\b} \left( -g^{\a \b} \d_{n}^{|m|} \Gamma^{\g}_{n \g} + \frac{1}{2} g^{\a \b} \d^{|m|}_n \cD - \frac{1}{4} \d_n^{|m|+1} g^{\a \b} \right) + \cO_{|m|}.
\end{equation}

\subsection{Determining the Taylor series of the metric at the boundary} \label{subsection: boundary determination for NT map}

In this section, we use our knowledge of the total symbol of $\Lambda$ computed in Section \ref{subsection: computation of total symbol of hodge DN}, as well as Equation \eqref{Sigma matrix equation}, to prove the following:

\begin{theorem}\label{Main Theorem for Beltrami Calderon}
    The full symbol of the Beltrami normal-to-tangential map $\Sigma$ determines all normal derivatives of the metric $g$ at the boundary.
\end{theorem}

\begin{proof}
    We consider the tangential and normal parts of Equation \eqref{Sigma matrix equation} separately, 
\begin{equation}\label{pre-symbol equations Sigma Lambda tang}
    \left( \Lambda^{tt} \right)^{\m}_{\ \n} \Sigma(f)_{\m} + \left( \Lambda^{tn} \right)_{\n} f = \frac{\lambda}{\sqrt{|g|}} g_{\n \a} \epsilon^{\a \b} \Sigma(f)_{\b} + \d_{\n} f,
\end{equation}
\begin{equation}\label{pre-symbol equations Sigma Lambda norm}
    \left( \Lambda^{nt} \right)^{\m} \Sigma(f)_{\m} + \left( \Lambda^{nn} \right) f = A^{\b} \Sigma(f)_{\b} + C(\d)^{\b} \Sigma(f)_{\b} +  \frac{1}{\lambda} \d_3 \left( \frac{1}{\sqrt{|g|}} \right) \epsilon^{\a \g}  \d_{\a} \Sigma(f)_{\g}.
\end{equation}
We take symbols of Equations \eqref{symbol equations Sigma Lambda tang} and \eqref{pre-symbol equations Sigma Lambda norm} to obtain
\begin{equation}\label{symbol equations Sigma Lambda tang}
    \left( b^{tt} \right)^{\m}_{\ \n} \, \# \, \sigma_{\m} + \left( b^{tn} \right)_{\n} = \frac{\lambda}{\sqrt{|g|}} g_{\n \a} \epsilon^{\a \b} \sigma_{\b} + i \x_{\n},
\end{equation}
\begin{equation}\label{symbol equations Sigma Lambda norm}
    \left( b^{nt} \right)^{\m} \, \# \, \sigma_{\m} +  b^{nn}  = A^{\b} \sigma_{\b} + C(i\x)^{\b} \, \# \, \sigma_{\b} +  \frac{i}{\lambda} \d_3 \left( \frac{1}{\sqrt{|g|}} \right) \epsilon^{\a \g}  \x_{\a} \, \# \, \sigma_{\g},
\end{equation}
where we have denoted the total symbol of $\Lambda$ by $b$ as in Section \ref{subsection: computation of total symbol of hodge DN}, the total symbol of $\Sigma$ by $\sigma$, and the composition of symbols by $\#$.

In this Section, we shall use the notation $\cO_m$ as in Section \ref{subsection: computation of total symbol of hodge DN} to denote any known quantity depending on normal derivatives of the metric up to order $m$, and tangential derivatives thereof. Here, we also include in $\cO_m$ any quantity depending on the symbol of $\Sigma$.

Observe first that from the principal symbol $\sigma_0$ of $\Sigma$, we can determine the metric at the boundary. Therefore, since $A^{\b}$ and $C(i \x)^{\b}$ are $\cO_0$, and since
\begin{equation*}
    \x_{\a} \, \# \, \sigma_{\g} = \x_{\a} \sigma_{\g} - i \d_{\a} \sigma_{\g},
\end{equation*}
we can re-write Equations \eqref{symbol equations Sigma Lambda tang}--\eqref{symbol equations Sigma Lambda norm} as
\begin{equation}\label{Tangentian equation O_0}
    \left( b^{tt} \right)^{\m}_{\ \n} \, \# \, \sigma_{\m} + \left( b^{tn} \right)_{\n} = \cO_0,
\end{equation}
\begin{equation}\label{Normal equation O_0}
     \left( b^{nt} \right)^{\m} \, \# \, \sigma_{\m} +  b^{nn} = \cO_0 + \frac{i}{\lambda} \d_3 \left( \frac{1}{\sqrt{|g|}} \right) \x_{\a} \epsilon^{\a \g} \sigma_{\g} + \frac{1}{\lambda} \d_3 \left( \frac{1}{\sqrt{|g|}} \right) \epsilon^{\a \g} \d_{\a} \sigma_{\g}. 
\end{equation}
Note that by the definition of $\cD$ \eqref{defining cD}, we have
\begin{equation*}
    \d_3 \left( \frac{1}{\sqrt{|g|}} \right) = \frac{1}{2 |g|} \cD.
\end{equation*}
Now, the degree $0$ part of Equations \eqref{Tangentian equation O_0}--\eqref{Normal equation O_0} can be written as
\begin{equation}\label{Tangential degree 1}
    \left( b^{tt}_{1} \right)^{\m}_{\ \n} \left( \sigma_{-1} \right)_{\m} + \left( b^{tt}_{0} \right)^{\m}_{\ \n} \left( \sigma_0 \right)_{\m} + \left( b_0^{tn} \right)_{\n} = \cO_0,
\end{equation}
\begin{equation}\label{Normal degree 1}
    \left( b_{0}^{nt} \right)^{\m} \left( \sigma_0 \right)_{\m} + b^{nn}_0 = \frac{i}{2 \lambda |g| } \cD \x_{\a} \epsilon^{\a \g} \left( \sigma_{-1} \right)_{\g} + \frac{1}{2 \lambda |g|} \cD \epsilon^{\a \g} \d_{\a} \left( \sigma_0 \right)_\g + \cO_0.
\end{equation}
Note that in Equation \eqref{Tangential degree 1}, we have explicitly written the term $b_1^{tt} \cdot \sigma_{-1}$, even though it is $\cO_0$. This is so that we can solve for $\sigma_{-1}$,
\begin{equation}\label{solving for sigma_1}
    \left( \sigma_{-1} \right)_{\n} = -\frac{1}{|\x|} \left( \left( b^{tt}_{0} \right)^{\m}_{\ \n} \left( \sigma_0 \right)_{\m} + \left( b_0^{tn} \right)_{\n} \right) + \cO_0.
\end{equation}
From Equation \eqref{b_0^tt}, we see that
\begin{equation}\label{b_0^tt O_0}
    \left( b_0^{tt} \right)^{\m}_{\n} =  \omega_\a \omega_\b \left( g^{\a \b} \Gamma^{\m}_{n \n} + \frac{1}{2} \cD g^{\a \b} \delta^{\m}_{\n} - \frac{1}{4} \delta^{\m}_{\n} \d_ng^{\a \b} \right) + \cO_0.
\end{equation}
where we recall that we have set $\omega = |\x|^{-1} \x$. So using Equations \eqref{b_0^tt O_0} and \eqref{b_0^tn}, we compute the first term on the right-hand-side of Equation \eqref{Normal degree 1} to be
\begin{align}\label{first term on RHS of degree 1}
    \frac{i \cD}{2 \lambda |g|} \x_{\m} \epsilon^{\m \n} \left( \sigma_{-1} \right)_{\n} &= -\frac{i \cD}{2 \lambda |g|} \frac{\x_{\m}}{|\x|} \epsilon^{\m \gamma} \left( \left( b^{tt}_{0} \right)^{\rho}_{\ \n} \left( \sigma_0 \right)_{\rho} + \left( b_0^{tn} \right)_{\n} \right) + \cO_0 \nonumber \\
    &= \frac{\cD}{2 \lambda |g|} \omega_{\m} \omega_{\a} \omega_{\b} \omega_{\rho} \epsilon^{\m \n} \left( g^{\a \b} \Gamma^{\rho}_{n \n} + \frac{1}{2} \cD g^{\a \b} \delta^{\rho}_{\n} - \frac{1}{4} \delta^{\rho}_{\n} \d_n g^{\a \b} + g^{\rho \b} g^{\a \g} \Gamma^n_{\g \n} \right) + \cO_0 \nonumber \\
    &= \frac{\cD}{2 \lambda |g|} \omega_{\m} \omega_{\a} \omega_{\b} \omega_{\rho} \left( \epsilon^{\m \n} g^{\a \b} \Gamma^{\rho}_{n \n} + \frac{1}{2} \cD g^{\a \b} \epsilon^{\m \rho} - \frac{1}{4} \epsilon^{\m \rho} \d_n g^{\a \b} + \epsilon^{\m \n} g^{\rho \b} g^{\a \g} \Gamma^n_{\g \n} \right) + \cO_0
\end{align}
where in obtaining the last line, we have re-written Equation \eqref{b_0^tn} as
\begin{equation*}
    \left( b_0^{tn} \right)_\n = i \omega_\a g^{\a \g} \Gamma^n_{\g \n} = i \omega_{\a} \omega_{\m} \omega_{\b} g^{\m \b} g^{\a \g} \Gamma^n_{\g \n}
\end{equation*}
since $g^{\m \b} \omega_{\m} \omega_{\b} = 1$. Now, since
\begin{equation*}
    \d_{\a} \left( \sigma_0 \right)_{\g} = \d_{\a} \left( \frac{i \x_{\g}}{|\x|} \right) = -\frac{i}{2} \frac{1}{|\x|^3} \d_{\a}g^{\b \rho} \x_{\b} \x_{\rho} \x_{\g},
\end{equation*}
the second-term on the right-hand side of Equation \eqref{Normal degree 1} is equal to
\begin{equation}\label{cubic term}
    \Omega^{\b \rho \g} \omega_{\b} \omega_{\rho} \omega_{\g} := -\frac{i}{4 \lambda |g|} \cD \epsilon^{\a \g} \d_{\a}g^{\b \rho} \omega_{\b} \omega_{\rho} \omega_{\g}.
\end{equation}
Finally, we compute the left-hand-side of Equation \eqref{Normal degree 1} using Equations \eqref{b_0^nn} and \eqref{b_0^nt}, which yields
\begin{align}\label{LHS of degree 1}
    \left( b_{0}^{nt} \right)^{\b} \left( \sigma_0 \right)_{\b} + b^{nn}_0 &= \omega_{\a} \omega_{\b} \left( -g^{\a \g} \Gamma^{\b}_{n \g} + \frac{1}{2} g^{\a \b} \cD - \frac{1}{4} \d_n g^{\a \b} \right) + \cO_0 \nonumber \\
    &= \omega_{\a} \omega_{\b} \omega_{\m} \omega_{\rho} \left( -g^{\m \rho}g^{\a \g} \Gamma^{\b}_{n \g} + \frac{1}{2} g^{\m \rho} g^{\a \b} \cD - \frac{1}{4} g^{\m \rho} \d_n g^{\a \b} \right) + \cO_0,
\end{align}
where in obtaining the last line, we have again inserted $g^{\m \rho} \omega_{\m} \omega_{\rho} = 1$. Therefore, upon combining Equations \eqref{first term on RHS of degree 1}, \eqref{cubic term}, and \eqref{LHS of degree 1}, we can rewrite Equation \eqref{Normal degree 1} as
\begin{align}
    \cO_0 &= \omega_\a \omega_\b \omega_{\m} \omega_{\rho} \bigg( -g^{\m \rho}g^{\a \g} \Gamma^{\b}_{n \g} + \frac{1}{2} g^{\m \rho} g^{\a \b} \cD - \frac{1}{4} g^{\m \rho} \d_n g^{\a \b} - \frac{\cD}{2 \lambda |g|} \epsilon^{\m \n} g^{\a \b} \Gamma^{\rho}_{n \n} - \frac{\cD^2}{4 \lambda |g|} g^{\a \b} \epsilon^{\m \rho} \nonumber \\
    &\ \ \ \ \ \ \ + \frac{\cD}{8 \lambda |g|} \epsilon^{\m \rho} \d_n g^{\a \b} - \frac{\cD}{2\lambda |g|} \epsilon^{\m \n} g^{\rho \b} g^{\a \g} \Gamma^n_{\g \n} \bigg) - \Omega^{\b \rho \g} \omega_{\b} \omega_{\rho} \omega_{\g}.
\end{align}
By parity, we can recover the coefficients of $\omega_\a \omega_\b \omega_{\m} \omega_{\rho}$, and thus we have
\begin{align}
    \cO_0 &= -g^{\m \rho}g^{\a \g} \Gamma^{\b}_{n \g} + \frac{1}{2} g^{\m \rho} g^{\a \b} \cD - \frac{1}{4} g^{\m \rho} \d_n g^{\a \b} - \frac{\cD}{2 \lambda |g|} \epsilon^{\m \n} g^{\a \b} \Gamma^{\rho}_{n \n} - \frac{\cD^2}{4 \lambda |g|} g^{\a \b} \epsilon^{\m \rho} \nonumber \\
    &\ \ \ \ \ \ \ + \frac{\cD}{8 \lambda |g|} \epsilon^{\m \rho} \d_n g^{\a \b} - \frac{\cD}{2\lambda |g|} \epsilon^{\m \n} g^{\rho \b} g^{\a \g} \Gamma^n_{\g \n}.
\end{align}
Contracting this expression with $g_{\m \rho}$, and using $g_{\m \rho} g^{\m \rho} = 2$, $g_{\m \rho} \epsilon^{\m \rho} = 0$, and $g_{\m \rho} g^{\rho \beta} = \delta^{\b}_{\m}$, we obtain
\begin{equation*}
    -2g^{\a \g} \Gamma^{\b}_{n \g} + g^{\a \b} \cD - \frac{1}{2} \d_n g^{\a \b} - \frac{\cD}{2 \lambda |g|} g_{\m \rho} \epsilon^{\m \n} g^{\a \b} \Gamma^{\rho}_{n \n}  - \frac{\cD}{2\lambda |g|} \epsilon^{\b \n} g^{\a \g} \Gamma^n_{\g \n} = \cO_0.
\end{equation*}
For the first of the terms involving $\lambda$, we find
\begin{align}
    - \frac{\cD}{2 \lambda |g|} g_{\m \rho} \epsilon^{\m \n} g^{\a \b} \Gamma^{\rho}_{n \n} &= -\frac{\cD}{4 \lambda |g|} g^{\a \b} \epsilon^{\m \n} g_{\m \rho}  g^{\rho \g} \d_n g_{\g \n} \nonumber \\
    &= -\frac{\cD}{4 \lambda |g|} g^{\a \b} \epsilon^{\m \n} \delta^{\g}_{\m} \d_n g_{\g \n} \nonumber \\
    &= -\frac{\cD}{4 \lambda |g|} g^{\a \b} \epsilon^{\g \n} \d_n g_{\g \n} = 0,
\end{align}
owing to the symmetry of $\d_n g_{\g \n}$. We are thus left with
\begin{equation*}
    -2g^{\a \g} \Gamma^{\b}_{n \g} + g^{\a \b} \cD - \frac{1}{2} \d_n g^{\a \b}  - \frac{\cD}{2\lambda |g|} \epsilon^{\b \n} g^{\a \g} \Gamma^n_{\g \n} = \cO_0.
\end{equation*}
We now contract with $g_{\a \b}$ to obtain
\begin{align}\label{we almost have D}
    \cO_0 &= -2 \delta^{\g}_{\b} \Gamma^{\b}_{n \g} + 2 \cD - \frac{1}{2} g_{\a \b} \d_n g^{\a \b}  - \frac{\cD}{2\lambda |g|} \epsilon^{\b \n} \delta^{\g}_{\b} \Gamma^n_{\g \n} \nonumber \\
    &= -2 \Gamma^{\g}_{n \g} + 2 \cD - \frac{1}{2} g_{\a \b} \d_n g^{\a \b} - \frac{\cD}{2\lambda |g|} \epsilon^{\g \n} \Gamma^n_{\g \n} \nonumber \\
    &= -2 \Gamma^{\g}_{n \g} + 2 \cD - \frac{1}{2} g_{\a \b} \d_n g^{\a \b},
\end{align}
owing to the symmetry of $\Gamma^n_{\g \n}$ in its lower indices. Recalling the various expressions for $\cD$, Equation \eqref{we almost have D} reduces to
\begin{equation*}
    \cO_0 = 3 \cD,
\end{equation*}
so that the quantity $\cD$ is known in terms of $\cO_0$. Once this is the case, the right-hand-side of Equation \eqref{Normal degree 1} is known, and thus we may write
\begin{equation}\label{Normal degree 1 after D}
    \left( b_{0}^{nt} \right)^{\m} \left( \sigma_0 \right)_{\m} + b^{nn}_0 = \cO_0.
\end{equation}
Equation \eqref{LHS of degree 1} then gives us
\begin{equation}\label{normal reduced after D}
     -g^{\a \g} \Gamma^{\b}_{n \g} - \frac{1}{4} \d_n g^{\a \b} = \cO_0.
\end{equation}
Armed with knowledge of $\cD$, we now want to recover $\d_n g^{\a \b}$. To this end, we return to the tangential equation \eqref{Tangential degree 1}, which we can write as
\begin{equation}\label{Tangential degree 1 reduced after D take 1}
    \left( b^{tt}_{0} \right)^{\m}_{\ \n} \left( \sigma_0 \right)_{\m} + \left( b_0^{tn} \right)_{\n} = \cO_0.
\end{equation}
Using Equations \eqref{b_0^tt} and \eqref{b_0^tn}, Equation \eqref{Tangential degree 1 reduced after D take 1} becomes
\begin{equation*}
    i \omega_\a \omega_\b \omega_{\m} \left( g^{\a \b} \Gamma^{\m}_{n \n} + \frac{1}{2} \cD g^{\a \b} \delta^{\m}_{\n} - \frac{1}{4} \delta^{\m}_{\n} \d_ng^{\a \b} \right) + i \omega_{\a} g^{\a \g} \Gamma^n_{\g \n} = \cO_0.
\end{equation*}
Since $\cD$ is now known in terms of $\cO_0$, we can move it to the right-hand-side. We then combine terms using $g^{\a \b} \omega_{\a} \omega_{\b} = 1$ to get
\begin{equation*}
     i \omega_\a \omega_\b \omega_{\m} \left( g^{\a \b} \Gamma^{\m}_{n \n} - \frac{1}{4} \delta^{\m}_{\n} \d_ng^{\a \b} + g^{\b \m} g^{\a \g} \Gamma^n_{\g \n} \right) = \cO_0,
\end{equation*}
whence we can recover
\begin{equation*}
    g^{\a \b} \Gamma^{\m}_{n \n} - \frac{1}{4} \delta^{\m}_{\n} \d_ng^{\a \b} + g^{\b \m} g^{\a \g} \Gamma^n_{\g \n} = \cO_0.
\end{equation*}
Taking a trace over $(\m, \n)$, we obtain
\begin{align}
    \cO_0 &= g^{\a \b} \Gamma^{\m}_{n \m} - \frac{1}{2} \d_ng^{\a \b} + g^{\b \m} g^{\a \g} \Gamma^n_{\g \m} \nonumber \\
    &= -g^{\a \b} \cD - \frac{1}{2} \d_ng^{\a \b} + g^{\b \m} g^{\a \g} \Gamma^n_{\g \m}.
\end{align}
Again, since $\cD$ is known, we arrive at
\begin{equation}\label{tangential degree 1 reduced finally}
    -\frac{1}{2} \d_n g^{\a \b} + g^{\b \m} g^{\a \g} \Gamma^n_{\g \m} = \cO_0.
\end{equation}
In order to recover $\d_n g^{\a \b}$, we want to compare Equation \eqref{tangential degree 1 reduced finally} with \eqref{normal reduced after D}. Observe that
\begin{equation}\label{mystery term 1}
    g^{\b \m} g^{\a \g} \Gamma^n_{\g \m} = -\frac{1}{2} g^{\b \m} g^{\a \g} \d_n g_{\g \m}
\end{equation}
\begin{equation}\label{mystery term 2}
    -g^{\a \g} \Gamma^{\b}_{n \g} = - \frac{1}{2} g^{\a \g} g^{\b \m} \d_n g_{\m \g}
\end{equation}
So that the left-hand-sides of Equations \eqref{mystery term 1} and \eqref{mystery term 2} are equal. Therefore, we can subtract Equation \eqref{tangential degree 1 reduced finally} from Equation \eqref{normal reduced after D}, to obtain
\begin{equation*}
    \frac{1}{4} \d_n g^{\a \b} = \cO_0.
\end{equation*}
This completes the first step, in which we recover the first normal derivative of the metric at the boundary. We can hence obtain all expressions in $\cO_1$. We now consider the degree $-1$ part of Equations \eqref{symbol equations Sigma Lambda tang} and \eqref{symbol equations Sigma Lambda norm}, which become
\begin{equation}\label{Tangential O_1 with b}
    \left( b^{tt}_{-1} \right)^{\m}_{\ \n} \left( \sigma_{0} \right)_{\m} + \left( b^{tn}_{-1} \right)_{\n} = \cO_1,
\end{equation}
\begin{equation}\label{Normal O_1 with b}
    \left( b^{nt}_{-1} \right)^{\m} \left( \sigma_{0} \right)_{\m} +  b^{nn}_{-1} = \cO_1.
\end{equation}
Using the expression we derived for $b_{-1}$ modulo $\cO_1$, \eqref{b_-1 tt}--\eqref{b_-1 nn}, Equations \eqref{Tangential O_1 with b} and \eqref{Normal O_1 with b} become
\begin{align}
    \cO_1 &= \frac{1}{2 |\x|} i \omega_{\a} \omega_{\b} \omega_{\m} \bigg( g^{\a \b} \d_n \Gamma^{\m}_{n \n} + g^{\a \b}g^{\m \g} \d_n \Gamma^n_{\n \g} + g^{\a \b} \d_n \Gamma^{\m}_{n \n} + \frac{1}{2}  g^{\a \b}\delta^{\m}_{\n} \d_n \cD -\frac{1}{4} \delta^{\m}_{\n} \d_n^2 g^{\a \b} \nonumber \\
    &\ \ \ \ \ \ \ +  g^{\b \m} g^{\a \g} \d_n \Gamma^n_{\g \n} \bigg) \label{Tangential O_1 with omega}
\end{align}
\begin{align}\label{Normal O_1 with omega}
    \cO_1 &= \frac{1}{2 |\x|} \omega_{\a} \omega_{\b} \bigg( -g^{\a \b} \d_{n} \Gamma^{\g}_{n \g} + \frac{1}{2} g^{\a \b} \d_n \cD - \frac{1}{4} \d_n^2 g^{\a \b} - g^{\a \g} \d_n \Gamma^{\b}_{\g n} \bigg).
\end{align}
By contracting the coefficient of Equation \eqref{Normal O_1 with omega} with $g_{\a \b}$, we obtain
\begin{equation}\label{getting d_n D}
    4 \d_n \cD - \frac{1}{4} g_{\a \b} \d_n^2 g^{\a \b} = \cO_1.
\end{equation}
On the other hand, we have
\begin{equation*}
    \frac{1}{2} g_{\a \b} \d_n^2 g^{\a \b} = \frac{1}{2} \d_n (g_{\a \b} \d_n g^{\a \b}) + \cO_1 = \d_n \cD + \cO_1,
\end{equation*}
and therefore, Equation \eqref{getting d_n D} reduces to
\begin{equation*}
    \d_n \cD = \cO_1.
\end{equation*}
Now that $\d_n \cD$ is known in terms of $\cO_1$, the normal equation \eqref{Normal O_1 with omega} reduces to
\begin{align}\label{normal O_1 to be cancelled}
    \cO_1 &= -\frac{1}{4} \d_n^2 g^{\a \b} - g^{\a \g} \d_n \Gamma^{\b}_{n \g} \nonumber \\
    &= -\frac{1}{4} \d_n^2 g^{\a \b} - \d_n \left( g^{\a \g} \Gamma^{\b}_{n \g} \right) + \cO_1
\end{align}
Armed with this, we extract the coefficient of the tangential equation \eqref{Tangential O_1 with omega} and take a trace over $(\mu, \nu)$ to obtain
\begin{align}
    \cO_1 &= -2 g^{\a \b} \d_n \cD + g^{\a \b} g^{\m \g} \d_n \Gamma^n_{\g \m} + g^{\a \b} \d_n \cD - \frac{1}{2} \d_n^2 g^{\a \b} + g^{\b \m} g^{\a \g} \d_n \Gamma^n_{\g \m} \nonumber \\
    &= -2 g^{\a \b} \d_n \cD + g^{\a \b} \d_n \cD + g^{\a \b} \d_n \cD - \frac{1}{2} \d_n^2 g^{\a \b} + g^{\b \m} g^{\a \g} \d_n \Gamma^n_{\g \m} \nonumber \\
    &= -\frac{1}{2} \d_n^2 g^{\a \b} + g^{\b \m} g^{\a \g} \d_n \Gamma^n_{\g \m} \nonumber \\
    &= -\frac{1}{2} \d_n^2 g^{\a \b} + \d_n \left( g^{\b \m} g^{\a \g} \Gamma^n_{\g \m} \right) + \cO_1. \label{trace over tangential O_1}
\end{align}
By Equations \eqref{mystery term 1} and \eqref{mystery term 2}, we can once again subtract Equations \eqref{trace over tangential O_1} and \eqref{normal O_1 to be cancelled} to get
\begin{equation*}
    \frac{1}{4} \d_n^2 g^{\a \b} = \cO_1.
\end{equation*}
We have thus recovered all second-order normal derivatives of the metric at the boundary. We may now conduct the inductive step in analogy with above. Thus, suppose we have recovered $\cO_{|m|}$. Taking the degree $m$ part of Equations \eqref{symbol equations Sigma Lambda tang} and \eqref{symbol equations Sigma Lambda norm}, we obtain
\begin{equation}\label{tangential O_m with b}
    \left( b_m^{tt} \right)^{\m}_{\ \n} \left( \sigma_0 \right)_{\m} + \left( b_m^{tn} \right)_{\n} = \cO_{|m|},
\end{equation}
\begin{equation}\label{normal O_m with b}
    \left( b_m^{nt} \right)^{\m} \left( \sigma_0 \right)_{\m} + b_m^{nn} = \cO_{|m|}.
\end{equation}
Using Equations \eqref{b_-m tt}--\eqref{b_-m nn}, we obtain
\begin{align}
    \cO_{|m|} &= \frac{i(-1)^{|m|+1}}{2^{|m|} |\x|^{|m|}} \omega_{\a} \omega_{\b} \omega_{\m} \bigg( g^{\a \b} \d^{|m|}_n \Gamma^{\m}_{n \n} + g^{\a \b}g^{\m \g} \d^{|m|}_n \Gamma^n_{\n \g} + g^{\a \b} \d^{|m|}_n \Gamma^{\m}_{n \n} + \frac{1}{2}  g^{\a \b}\delta^{\m}_{\n} \d^{|m|}_n \cD \nonumber \\
    &\ \ \ \ \ \ \ -\frac{1}{4} \delta^{\m}_{\n} \d_n^{|m|+1} g^{\a \b} + g^{\a \g} g^{\b \m} \d_n^{|m|} \Gamma^n_{\g \n} \bigg) + \cO_{|m|}, \label{tangential with O_m omega}
\end{align}
\begin{equation}\label{normal with O_m omega}
    \cO_{|m|} = \frac{(-1)^{|m|+1}}{2^{|m|} |\x|^{|m|}} \omega_{\a} \omega_{\b} \left( -g^{\a \b} \d_{n}^{|m|} \Gamma^{\g}_{n \g} + \frac{1}{2} g^{\a \b} \d^{|m|}_n \cD - \frac{1}{4} \d_n^{|m|+1} g^{\a \b} - g^{\a \g} \d_n^{|m|} \Gamma^{\b}_{\g n} \right) + \cO_{|m|}.
\end{equation}
Contracting the coefficient of Equation \eqref{normal with O_m omega} with $g_{\a \b}$, we obtain 
\begin{align}
    \cO_{|m|} &= 4 \d_n^{|m|} \cD - \frac{1}{4} g_{\a \b} \d_n^{|m|+1}g^{\a \b} \nonumber \\
    &= 4 \d_n^{|m|} \cD - \frac{1}{4} \d_n^{|m|} \left( g_{\a \b} \d_n g^{\a \b} \right) + \cO_{|m|},
\end{align}
and so we are left with $\d_n^{|m|} \cD = \cO_{|m|}$. Equation \eqref{normal with O_m omega} then reduces to
\begin{equation}\label{last subtract normal}
    -\frac{1}{4} \d_n^{|m|+1} g^{\a \b} - g^{\a \g} \d_n^{|m|} \Gamma^{\b}_{\g n} = \cO_{|m|},
\end{equation}
while after taking a trace over $(\m \n)$, Equation \eqref{tangential O_m with b} reduces to
\begin{equation}\label{last subtract tangent}
    -\frac{1}{2} \d_n^{|m|+1} g^{\a \b} + g^{\a \g} g^{\b \m} \d_n^{|m|} \Gamma^n_{\g \m} = \cO_{|m|}.
\end{equation}
As before, we can add subtract Equations \eqref{last subtract normal} and \eqref{last subtract tangent} to obtain
\begin{equation*}
    \d_n^{|m|+1}g^{\a \b} = \cO_{|m|}.
\end{equation*}
This completes the inductive step, and thus also the proof that the symbol of $\Sigma$ determines the normal derivatives of $g$ at the boundary.
\end{proof}

\begin{remark}\label{remark: boundary determination for harmonic fields}
    Following up on Remark \ref{remark: introducing NT map for harmonic}, we note that the above proof is easily adapted to the case of harmonic fields, for which $\lambda = 0$ and the divergence-free equation is prescribed, rather than a consequence of the Beltrami field equation. Indeed, despite the presence of $\lambda^{-1}$ in the second line of Equation \eqref{Sigma matrix equation}, that line is easily seen to follow from $d^* u = 0$; therefore $\lambda^{-1}$ may be eliminated and the Beltrami field equation is in fact not needed there. The rest of the proof then applies with minimal changes to the case of the normal-to-tangential map for Harmonic fields.
\end{remark}

\section{Recovering a real-analytic simply connected manifold from the Beltrami normal-to-tangential map}\label{sec: recovering manifold}

In this section, we shall recover a compact simply connected real-analytic Riemannian $3$-manifold with boundary from its Beltrami normal-to-tangential map, thus proving Theorem \ref{main reconstruction theorem}. For this section, we shall assume that $(M_1,g_1)$ and $(M_2,g_2)$ are two such manifolds, whose boundaries are identified, and that their corresponding Beltrami normal-to-tangential maps $\Sigma_1$ and $\Sigma_2$ are equal. As the boundaries are identified, we will sometimes denote $\d M_1$ and $\d M_2$ simply by $\d M$.

By Theorem \ref{Main Theorem for Beltrami Calderon}, the Taylor series of the real-analytic metrics $g_1$ and $g_2$ are equal on $\d M$. Therefore, by attaching $\d M \times [-\delta_0,0]$ to $M_i$ along the boundary $\d M$ as in  , we can extend the manifold $M_i$ and the metric $g_i$ to a larger real-analytic Riemannian manifold with boundary $(\tilde{M}_i, \tilde{g}_i)$ such that $\tilde{M}_1 \setminus M_1$ is identified with $\tilde{M}_2 \setminus M_2$, and the extended metrics $\tilde{g}_1$ and $\tilde{g}_2$ agree in that region. We shall assume this setup for the remainder of the paper.

\subsection{The $b$-fields}\label{subsec: analogue of Greens functions}

To this end, drawing on the idea pursued in \cite{lassas2003} and \cite{Krupchyk2011}, we shall define an appropriate analogue of Green's forms for the boundary value problem \eqref{Beltrami source problem}, which depend analytically on a set of parameters, and which shall be used to embed our manifold into a suitable Sobolev space. Since any inhomogeneous term in the boundary value problem \eqref{Beltrami source problem} must be divergence-free in order to guarantee the existence of solutions, we cannot simply solve the problem with delta functions, and so there does not exist a Green's functions for the Beltrami problem. In lieu of delta functions, therefore, we use small ``loop sources", which we define momentarily. We call the corresponding solutions {\em $b$-fields}, as in the case where $\lambda = 0$, these correspond to magnetic fields arising from small loops of current. The parameters that these electric current loops, and hence the $b$-fields, depend on are therefore the location and orientation of the loops.

Let $S \tilde{M}_i$ denote the sphere bundle of $\tilde{M}_i$,
\begin{equation*}
    S\tilde{M}_i := \{ \omega \in T\tilde{M}_i \ | \ | \omega | = 1 \},
\end{equation*}
and let $\pi : S\tilde{M}_i \to \tilde{M}_i$ be the canonical projection. For $\varepsilon > 0$, let
\begin{equation*}
    \tilde{M}_i^{\varepsilon} := \{ x \in \tilde{M}_i \ | \ d(x, \d \tilde{M}_i) \geq \varepsilon \}.
\end{equation*}
For any $\omega \in S\tilde{M}_i^{\varepsilon}$ and $0 < \delta < \varepsilon$, we define the distributional current loop $J^i_{\omega, \delta}$ generated by $\omega$ with radius $\delta$ as follows. Let $\a_{\omega, \delta} : [0,2\pi] \to T_{\pi(\omega)}\tilde{M}_i$ parametrize an oriented circle of radius $\delta$ in the plane orthogonal to $\omega$. That is, if $(e_1,e_2,\omega)$ form an oriented orthonormal basis for $T_{\pi(\omega)}\tilde{M}_i$, then
\begin{equation*}
    \a_{\omega, \delta}(t) := \delta \left( (\cos{t}) e_1 + (\sin{t}) e_2 \right).
\end{equation*}
Let $\gamma_{\omega, \delta} : [0, 2\pi] \to \tilde{M}_i$ be the curve obtained by mapping $\a_{\omega, \delta}$ into $\tilde{M}_i$ via the exponential map,
\begin{equation*}
    \gamma_{\omega, \delta}(t) := \exp_{\pi(\omega)}\left( \a_{\omega, \delta}(t) \right).
\end{equation*}
We then define the distributional current loop $J^i_{\omega, \delta}$ by its action on a smooth $1$-form $A$,
\begin{equation}\label{current loop definition}
    J^i_{\omega, \delta}(A) := \int_{\gamma_{\omega, \delta}} \gamma^*_{\omega, \delta} A.
\end{equation}
We note the following properties of $J^i_{\omega, \delta}$. First, it does not depend on the orthonormal basis $(e_1,e_2,\omega)$. Second, $J^i_{\omega, \delta}$ is a distributional vector field with support and singular support on the image of $\gamma_{\omega,\delta}$, which we shall also denote by $\l(\omega, \delta)$. Third, the current loop $J^i_{\omega,\delta}$ is divergence-free in the sense of distributions. And finally, as one can easily verify using equation \eqref{current loop definition}, we have $J_{\omega, \delta} \in H^s$ for $s < -\frac{3}{2}$, and the map $(\omega, \delta) \mapsto J_{\omega, \delta}$ is $C^1$ into $H^{s'}$ for $s' < -\frac{3}{2} - 1$.

The existence theory for Beltrami fields given in \cite{enciso2015} generalizes to divergence-free distributions and allows us to solve the following boundary value problem, so long as $\lambda$ is not a Beltrami singular value of $\tilde{M}_i$, which we can arrange in the construction of the extended manifolds $\tilde{M}_i$ since $\lambda$ is not a Beltrami singular value of $M_1$ and $M_2$ by assumption. We thus solve
\begin{equation}\label{BVP for b-fields}
    \begin{cases}
        \curl_{g_i} {b^i_{\omega, \delta}} - \lambda b^i_{\omega, \delta} = J^i_{\omega, \delta} \ \text{in} \ \tilde{M}_i \\
        \tilde{\nu} \cdot b^i_{\omega, \delta} |_{\d \tilde{M}} = 0, \\
        \cH^t_{\tilde{M}_i}(b^i_{\omega, \delta}) = 0,
    \end{cases}
\end{equation}
for any $\omega \in SM_i^{\varepsilon}$ and $0 < \delta < \e$. Note that since $M_i$ is simply connected, we can arrange that $\tilde{M}_i$ be simply connected also, in which case the condition that the harmonic part be zero becomes superfluous. We will henceforth assume this is the case.

We note that by elliptic regularity, $b^i_{\omega, \delta}$ has singular support equal to that of $J_{\omega, \delta}$, namely the circle $\l(\omega, \delta)$ of radius $\delta$ generated by $\omega$. We also note that $b^i_{\omega, \delta}(x)$ is real-analytic in $x$ up to the boundary of $\tilde{M}_i$ \cite[Appendix A]{Enciso2023}, and in the parameters $(\omega, \delta)$, whenever $x$ is away from the singular set $\l(\omega, \delta)$. We also mention here a useful representation of the $b$-fields in terms of an integral kernel. To this end, we refer to \cite{enciso2018}, where an integral kernel for the curl operator on a manifold with boundary is constructed, and note that the methods applied therein may be straightforwardly generalized to the case of the Beltrami operator where $\lambda$ is non-zero. We thus obtain the following extension of \cite[Theorem 1.1]{enciso2018}: 

\begin{theorem}\label{theorem: existence of integral kernel}
    Let $(M,g)$ be as above. There exists an integral kernel $K_{\lambda}(x,y) \in \End(T_yM, T_xM)$ that is smooth outside the diagonal and satisfies
\begin{equation}\label{asymptotics of kernel}
    |K_{\lambda}(x,y)| \leq \frac{C(y)}{d(x,y)^2},
\end{equation}
    where $C(y)$ depends only on the distance from $y$ to the boundary, such that for any $H^k$ divergence-free vector field $v$ on $M$, the unique solution $u$ to the boundary value problem
    \begin{equation} \label{BVP for integral kernel}
        \begin{cases}
        \curl_{g} u - \lambda u = v \ \text{in} \ M, \\
        \nu \cdot u |_{\d M} = 0, \\
        \cH^t_{M}(u) = 0.
    \end{cases}
    \end{equation}
    can be represented as an integral of the form
    \begin{equation*}
        u(x) = \int_M K_{\lambda}(x,y) v(y)\, dy.
    \end{equation*}    
\end{theorem}

\begin{remark}
    As observed in \cite{enciso2018}, while Theorem \ref{theorem: existence of integral kernel} yields the existence of an integral kernel with the above properties, this kernel is easily seen to be non-unique. The solution to the boundary value problem \eqref{BVP for integral kernel}, however, {\em is} unique.
\end{remark}

With the integral kernel of Theorem \ref{theorem: existence of integral kernel}, the $b$-fields defined above as the solutions to \eqref{BVP for b-fields} can then be represented distributionally as
    \begin{equation}\label{integral representation of b-fields}
        b^i_{\omega, \delta}(x) := \int_{\gamma_{\omega, \delta}} K_{\lambda}(x, \gamma_{\omega, \delta}(t)) \gamma_{\omega,\delta}'(t)\, dt,
    \end{equation}
as one can easily verify. We can now use the representation \eqref{integral representation of b-fields} to deduce the following result, which shows that like the Green's functions for the Laplacian, the $b$-fields have a simple asymptotic form near singularities:

\begin{proposition}\label{proposition asymptotics}
    As $x \rightarrow \l(\omega, \delta)$, the $b$-fields have the asymptotic behaviour
    \begin{equation}\label{asymptotics}
        |b^i_{\omega, \delta}(x)| \sim \frac{C}{d(x, \l(\omega, \delta))}.
    \end{equation}
    In particular, the $b$-fields are $L^1_{\loc}$.
\end{proposition}
    
\begin{proof}
    We have \eqref{integral representation of b-fields} for $x \notin \l(\omega,\delta)$. So, choosing a coordinate $t$ along $\gamma_{\omega, \delta}$ centred on the point closest to $x$, we have, by the asymptotics \eqref{asymptotics of kernel},
    \begin{align}
        |b^i_{\omega, \delta}(x)| &\leq  \int_{\gamma_{\omega, \delta}} |K_{\lambda}(x, \gamma_{\omega, \delta}(t))|\, dt \nonumber \\
        &\leq \int_{\gamma_{\omega, \delta}} \frac{C_0}{d(x, \gamma_{\omega, \delta}(t))^2} \, dt \nonumber \\
        &\leq C_1 + \int_{-\epsilon}^{\epsilon} \frac{C_2}{d(x,\l(\omega, \delta))^2 + t^2} \, dt \nonumber \\
        &\leq \frac{C}{d(x,\l(\omega, \delta))},
    \end{align}
    yielding the desired asymptotics \eqref{asymptotics}.
\end{proof}

Having established the relevant properties of the $b$-fields, we now want to show, as in \cite[Lemma 3.4]{Krupchyk2011}, that the Beltrami normal-to-tangential map determines the $b$-field in the outer region of the extended manifold where the two metrics were constructed to agree, which shall be an important result in the next section. Thus, as above, fix $\delta > 0$ such that $\delta < \e$, let $U := \tilde{M}_i^{\varepsilon} \setm M_i$, and let $U_\delta := \{ p \in U \ | \ d(p, \d M_i) > \delta \}$. We have the following result:

\begin{proposition}\label{proposition: b-fields agree near boundary}
    We have $b^1_{\omega, \delta}(x) = b^2_{\omega, \delta}(x)$ for all $(x,\omega) \in U \times SU_{\delta}$ such that $x \notin \l(\omega, \delta)$.
\end{proposition}

\begin{proof}
    Fix $\omega \in SU_{\delta}$, and let $\b$ be the unique vector field on $M_2$ solving the boundary value problem
    \begin{equation*}
        \begin{cases}
            \curl_{g_2}{\b} - \lambda \b = 0, \\
            \n \cdot \b |_{\d M} = \n \cdot b^1_{\omega} |_{\d M}, \\
            \cH^t_{M_2}(\b) = 0.
        \end{cases}
    \end{equation*}
    Note that since $M_i$ is simply connected, \eqref{BVP for b-fields} implies that $b^1_{\omega, \delta}$ is the unique solution $\gamma$ to
    \begin{equation}\label{BVP for gamma}
    \begin{cases}
        \curl_{g_1} \gamma - \lambda \gamma = 0, \\
        \n \cdot \gamma |_{\d M} =  \n \cdot b^1_{\omega, \delta} |_{\d M}, \\
        \cH^t_{M_1}(\gamma) = 0, 
    \end{cases}
\end{equation}
inside $M_1$. Therefore, since $\Sigma_1 = \Sigma_2$, we have
\begin{equation*}
    \left( \b|_{\d M} \right)^t = \Sigma_2 \left( \n \cdot \b|_{\d M} \right) = \Sigma_2 \left( \n \cdot b^1_{\omega, \delta}|_{\d M} \right) = \Sigma_1 \left( \n \cdot b^1_{\omega, \delta}|_{\d M} \right) = \left( b^1_{\omega, \delta} |_{\d M}\right)^t,
\end{equation*}
so that in fact, we have $\b|_{\d M} = b^1_{\omega, \delta} |_{\d M}$. We may thus define a vector field $\tilde{\b}$ on $\tilde{M}_2$ by
\begin{equation*}
    \tilde{\b} := \begin{cases}
        \b & \mathrm{in} \ M_2, \\ b^1_{\omega, \delta} & \mathrm{in} \ U,
    \end{cases}
\end{equation*}
which by construction satisfies the boundary value problem
\begin{equation}\label{BVP for beta tilde}
    \begin{cases}
        \curl_{g_2} \tilde{\b} - \lambda \tilde{\b} = J^2_{\omega, \delta}, \\
        \n \cdot \tilde{\b} |_{\d \tilde{M}} = 0, \\
        \cH^t_{\tilde{M}_2}(\tilde{\b}) = 0,
    \end{cases}
\end{equation}
since the extended metrics $\tilde{g}_1$ and $\tilde{g}_2$ agree in $U$. Therefore, by the uniqueness of solutions to \eqref{BVP for beta tilde}, it follows that $\tilde{\b} = b^2_{\omega, \delta}$ in $\tilde{M}_2$. In particular, we have that $b^2_{\omega, \delta}$ coincides with $b^1_{\omega, \delta}$ in $U$.
\end{proof}

\begin{remark}\label{remark on simply connectedness}
    Note that if $\tilde{M}_1$ is not simply connected, then the vanishing of the harmonic part of $b^1_{\omega, \delta}$ on $\tilde{M}_1$ does not guarantee the vanishing of its harmonic part when restricted to $M_1$, and so we cannot conclude that $b^1_{\omega, \delta}$ is the unique solution to \eqref{BVP for gamma} in that case. Similarly, if $M_2$ is not simply connected, then the vanishing of $\cH^t_{M_2}(\b)$ does not guarantee the vanishing of $\cH^t_{\tilde{M}_2}(\tilde{\b})$, and so we would not be able to conclude that $\tilde{b}$ coincides with $b^2_{\omega, \delta}$. One sees therefore why we need the manifolds to be simply connected in order to match the Cauchy data, and conclude that the Beltrami normal-to-tangential map determines the $b$-fields in the extended region. 
\end{remark}

\subsection{Embedding $\tilde{M}_i^{\varepsilon}$ into a Sobolev space and constructing the isometry}\label{subsec: embedding M into H^s}

Having established the necessary properties of the $b$-fields, we now proceed with the approach given in \cite{Krupchyk2011}, using the $b$-fields in lieu of the Green's forms and emphasizing the necessary adaptations. For the remainder of this paper, we shall fix a loop radius $\delta > 0$ with $\delta < \e$ as above, and henceforth omit the dependence on $\delta$ in our notation.

We begin by choosing an open set $\tilde{U} \cc U_{\delta}$. For some fixed $s < -\frac{3}{2}$, we can define the maps
\begin{align}\label{definition of embedding maps}
\begin{split}
    \sB_i : S\tilde{M}_i^{\e} &\to H^s(\tilde{U}, T\tilde{U}) \\
    \omega &\mapsto b^i_{\omega}|_{\tilde{U}}.
\end{split}
\end{align}
We will eventually show that $\sB_1$ and $\sB_2$ are embeddings with the same image. This will imply that $S \tilde{M}_1^{\e}$ is isometric to $S \tilde{M}_2^{\e}$, from which the isometry of $M_1$ and $M_2$ shall follow. We have

\begin{proposition}
    For $s < -\frac{3}{2}$, the map $\sB_i$ is a $C^1$-map into $H^s(\tilde{U}, T\tilde{U})$.
\end{proposition}

\begin{proof}
    We want to show that the map $\omega \mapsto b^i_{\omega}$ is $C^1$ into $H^s$. To this end, note that by \eqref{BVP for b-fields} the derivatives of the $b$-fields with respect to $\omega$ satisfy the boundary value problem
    \begin{equation}\label{BVP for omega derivatives}
    \begin{cases}
        \curl_{g_i} {(\d_{\omega} b^i_{\omega})} - \lambda\, (\d_{\omega}b^i_{\omega}) = \d_{\omega}J^i_{\omega} \ \ \text{in} \ \ \tilde{M}_i \\
        \tilde{\nu} \cdot (\d_{\omega} b^i_{\omega}) |_{\d \tilde{M}} = 0.
    \end{cases}
\end{equation}
Since $\d_{\omega} J^i_{\omega}$ is easily seen to be in $H^{s'}$ for $s' < -\frac{5}{2}$, it follows by elliptic regularity that $\d_{\omega} b^i_{\omega} \in H^s$ for $s < -\frac{3}{2}$. Since $\d_{\omega}J^i_{\omega}$ depends continuously on $\omega$, so does $\d_{\omega} b^i_{\omega}$.
\end{proof}

Next we prove

\begin{proposition}\label{Proposition: B is embedding}
    The map $\sB_i$ is an embedding, and is real-analytic on $\tilde{M}_i^{\e} \setm V$, where $V$ is adequately chosen and satisfies $\tilde{U} \cc V \cc U_{\delta}$.
\end{proposition}

\begin{proof}
    The analyticity of $\sB_i$ away from $\tilde{U}$ follows from the analyticity of the $b$-fields away from singularities. Now, since $S \tilde{M}_i^{\e}$ is compact, to prove that $\sB_i$ is an embedding it suffices to show that is is an injective immersion. We first show that $\sB_i$ is injective. Indeed, suppose that $\sB_i(\omega) = \sB_i(\omega')$ for $\omega \neq \omega'$. Then we have $b^i_{\omega}(x) = b^i_{\omega'}(x)$ for all $x \in \tilde{U}$ away from singularities. But by analyticity, this must hold for all $x \in \tilde{M}_i^{\e} \setminus \left( \l(\omega) \cup \l(\omega') \right)$. If the singular circles $\l(\omega)$ and $\l(\omega')$ are not equal, then this contradicts the asymptotics of the $b$-fields. Therefore $\l(\omega) = \l(\omega')$, from which it easily follows that we must have $\omega = \omega'$. Hence, $\sB_i$ is injective.

    To show that $\sB_i$ is an immersion, suppose there exists $\omega_0 \in S \tilde{M}_i^{\e}$ for which $D \sB_i(\omega_0)$ is not injective. That is, there exists $v \in T_{\omega_0} S \tilde{M}^{\e}_i$ such that
    \begin{equation*}
        0 = \left( D \sB_i \right)_{\omega_0} v = v^{k} \d_{\omega^k} b_{\omega} |_{\omega_0}
    \end{equation*}
    in local coordinates about $\omega_0$. But this would imply that
    \begin{equation*}
        0 = v^k \d_{\omega^k} b_{\omega}|_{\omega_0}(x) 
    \end{equation*}
    for all $x \in \tilde{U} \setm \l(\omega_0)$. By real-analyticity, this would hold for all $x \in \tilde{M}^{\e}_i \setm \l(\omega_0)$, contradicting the asymptotics as $x \mapsto \l(\omega_0)$. Thus, $\sB_i$ is an injective immersion, hence an embedding.
\end{proof}

We are now ready to prove the main result of this section (c.f. \cite[Theorem 3.7]{Krupchyk2011}), from which Theorem \ref{main reconstruction theorem} will follow:

\begin{theorem}\label{isometry of sphere bundles theorem}
    Suppose that the $b$-fields satisfy
    \begin{equation*}
        b^1_{\omega}(x) = b^2_{\omega}(x)
    \end{equation*}
    for all $(x,\omega) \in U \times S U_{\delta}$ such that $x \notin \l(\omega, \delta)$. Then
    \begin{equation*}
        \sB_1(S \tilde{M}_1^{\e}) = \sB_2(S \tilde{M}_2^{\e}) \sub H^s(\tilde{U}, T\tilde{U}),
    \end{equation*}
    and $\sB_2^{-1} \sB_1 : S \tilde{M}_1^{\e} \to S \tilde{M}_2^{\e}$ is an isometry.
\end{theorem}

\begin{proof}
To prove this, following \cite{Krupchyk2011}, we introduce the following sets for a small $\epsilon_0 > 0$:
\begin{equation*}
    N(\epsilon_0) = \{ x \in \tilde{M}^{\e}_1 \ : \ d(x, \d \tilde{M}^{\e}_1) \leq \epsilon_0 \},
\end{equation*}
\begin{equation*}
    C(\epsilon_0) = \{ x \in \tilde{M}^{\e}_1 \ : \ d(x, \d \tilde{M}^{\e}_1) > \epsilon_0 \},
\end{equation*}
and we can further assume that $C(\epsilon_0)$ is connected. Let us then take some $\omega_0 \in SC(\epsilon_0)$, and let $V_1 \sub SC(\epsilon_0)$ be the largest open set containing $\omega_0$ such that $\sB_1(\omega) \in \sB_2(S \tilde{M}^{\e}_2)$ for all $\omega \in V_1$. We know that $V_1$ is non-empty because the $b$-fields agree in $U \times SU_{\delta}$. Since $\sB_2$ is injective by Proposition \ref{Proposition: B is embedding}, we can define the map
\begin{equation*}
    \hat{\Phi} = \sB_2^{-1} \sB_1 : V_1 \to S\tilde{M}_2^{\e}.
\end{equation*}
Let $D_1 \sub V_1$ be the largest connected open set such that $\hat{\Phi}$ is a real-analytic local diffeomorphism and local isometry on $D_1$. As in \cite{Krupchyk2011}, we want to show that $SN(\epsilon_0) \cup D_1 = S \tilde{M}^{\e}_1$ by contradiction. Thus, let $\omega_1$ be the point in $S \tilde{M}_1^{\e} \setm (SN(\epsilon_0) \cup D_1)$ closest to $\omega_0$, so that $\omega_1 \in \d D_1$. We now want to prove the following analogue of \cite[Lemma 3.8]{Krupchyk2011}:

\begin{lemma}\label{lemma sequences}
    There exist $\omega_2$ in the interior of $S \tilde{M}^{\e}_2$ such that
    \begin{equation*}
        \sB_2(\omega_2) = \sB_1(\omega_1) \in H^s(\tilde{U}, T\tilde{U}),
    \end{equation*}
    and a sequence $(\omega_j) \sub D_1$ such that $\omega_j \to \omega_1$ and $\hat{\Phi}(\omega_j) \to \omega_2$.
\end{lemma}

\begin{proof}
    Since $\omega_1 \in \d D_1$, we can take any sequence $(\omega_j) \sub D_1$ such that $\omega_j \to \omega_1$. Since $\hat{\Phi}$ is a local isometry on $D_1$, there is a sequence $(\x_j) \sub S \tilde{M}_2^{\e}$ such that $\sB_1(\omega_j) = \sB_2(\x_j)$. By compactness, $(\xi_j)$ has a convergent subsequence. If it has a convergent subsequence that converges to a point in the interior of $S \tilde{M}_2^{\e}$, then we are done. So let us suppose that every convergent subsequence of $(\xi_j)$ converges to a point $\x_0 \in \d S\tilde{M}_2^{\e}$. We then would have, for fixed $y \in \tilde{U}$,
    \begin{equation*}
        \sB_1(\omega_j)(y) = \sB_2(\x_j)(y) \to \sB_2(\x_0)(y) = b^2_{\x_0}(y).
    \end{equation*}
    On the other hand,
    \begin{equation*}
        \sB_1(\omega_j)(y) \to \sB_1(\omega_1)(y) = b^1_{\omega_1}(y).
    \end{equation*}
    Since $\x_0$ is on the boundary of $S \tilde{M}_2^{\e}$, Proposition \ref{proposition: b-fields agree near boundary} and the above yield
    \begin{equation}\label{lemma sequences important equation}
        b^1_{\omega_1}(y) = b^2_{\x_0}(y) = b^1_{\x_0}(y).
    \end{equation}
    This is true for any $y \in \tilde{U}$ away from singularities, and so therefore, by analytic continuation, it must hold everywhere.  But $\omega_1$ is in the interior of $S \tilde{M}^{\e}_1$, while $\x_0$ is on the boundary. This contradicts the asymptotics of the $b$-fields. Therefore, $(\x_j)$ has a convergent subsequence that converges to an interior point $\omega_2$ of $S\tilde{M}^{\e}_2$, which proves the lemma. 
\end{proof}

Now we assume that $\tilde{U} \cc U_{\delta}$ is chosen so that $\sB_i$ is an analytic embedding in a neighbourhood of $\omega_i$, for $i = 1,2$. We also have $S\tilde{U} \sub D_1$. 

\begin{lemma}\label{lemma tangent spaces equal}
    We have
    \begin{equation*}
        D \sB_1(\omega_1)(T_{\omega_1} S\tilde{M}_1^{\e}) = D \sB_2(\omega_2)(T_{\omega_2} S\tilde{M}_2^{\e}) \sub H^s(\tilde{U}, T\tilde{U}).
    \end{equation*}
\end{lemma}

\begin{proof}
    By Lemma \ref{lemma sequences}, there is a sequence $(\omega_j) \sub D_1$ with $\omega_j \to \omega_1$ and $\hat{\Phi}(\omega_j) \to \omega_2$. By definition of $D_1$, the maps $\sB_1$ and $\sB_2 \circ \hat{\Phi}$ coincide on $D_1$. Hence, so do their derivatives. Therefore,
    \begin{equation*}
        D \sB_1 (\omega_j)(T_{\omega_j} S\tilde{M}_1^{\e}) = D \sB_2 (\hat{\Phi}(\omega_j))(T_{\hat{\Phi}(\omega_j)} S \tilde{M}_2^{\e}).
    \end{equation*}
    By the choice of $s$, the derivatives $D\sB_i$ are continuous on $S \tilde{M}_i^{\e}$. The result follows. 
\end{proof}

Let us now set 

\begin{equation*}
    \sV := D \sB_1(\omega_1)(T_{\omega_1} S\tilde{M}_1^{\e}) = D \sB_2(\omega_2)(T_{\omega_2} S\tilde{M}_2^{\e}) \sub H^s(\tilde{U}, T\tilde{U}).
\end{equation*}

Since $D\sB_i$ is injective, we have that $\dim{\sV} \cong \dim{T_{\omega_i} S \tilde{M}_i^{\e}} = 5$. Let $P : H^s(\tilde{U}, T\tilde{U}) \to \sV$ be the orthogonal projection onto $\sV$. Then
\begin{equation*}
    D(P\sB_j)(\omega_j) = P(D\sB_j)(\omega_j) = D\sB_j(\omega_j),
\end{equation*}
which is bijective onto $\sV$. So by the inverse function theorem,
\begin{equation*}
    P\sB_i : \sN(\omega_i, S\tilde{M}_i^{\e}) \to \sN(Pu, \sV),
\end{equation*}
is a real-analytic diffeomorphism, where $\sN(p,W)$ denotes a neighbourhood of $p$ in $W$. Now, by decomposing $H^s(\tilde{U}, T\tilde{U}) = \Im{P} \oplus \Im(1-P)$, it is easy to see that the sets
\begin{equation*}
    \{ \sB_i(\omega) \ | \ \omega \in \sN(\omega_i, S \tilde{M}_i^{\e}) \}
\end{equation*}
are the graphs of the real-analytic functions 
\begin{equation}\label{(1-P)B(PB)-1 function}
    (1-P) \sB_i (P \sB_i)^{-1} : \sN(Pu, \sV) \to \Im(1-P).
\end{equation}
Moreover, using Lemma \ref{lemma sequences}, the fact that $\hat{\Phi}$ is a local isometry on $D_1$, and the real-analyticity of the functions \eqref{(1-P)B(PB)-1 function}, one easily establishes that
\begin{equation}\label{(1-P)B(PB)-1 equality}
        (1-P)\sB_1(P\sB_1)^{-1}(v) = (1-P)\sB_2(P\sB_2)^{-1}(v)
\end{equation}
holds for all $v \in \sN(Pu, \sV)$.

We are now ready to prove that $\omega_1$ is an interior point of $D_1$. We will do this by showing that for every $\tilde{\omega}_1 \in \sN(\omega_1, S \tilde{M}_1^{\e})$, there exists a unique $\tilde{\omega}_2 \in \sN(\omega_2, S \tilde{M}_2^{\e})$ such that
\begin{equation}\label{B_1 = B_2 in neighbourhood}
    \sB_1(\tilde{\omega}_1) = \sB_2(\tilde{\omega}_2).
\end{equation}
Indeed, note that we may take $\tilde{\omega_2} = (P \sB_2)^{-1} (P \sB_1)(\tilde{\omega}_1)$. Then $P\sB_1(\tilde{\omega}_1) = P\sB_2(\tilde{\omega}_2)$ holds by definition, while equation \eqref{(1-P)B(PB)-1 equality} implies
\begin{equation*}
    (1-P)\sB_1(\tilde{\omega}_1) = (1-P)\sB_2(\tilde{\omega}_2).
\end{equation*}
Therefore \eqref{B_1 = B_2 in neighbourhood} holds, and so the map $\hat{\Phi} = \sB_2^{-1}\sB_1$ is real-analytic in $\sN(\omega_1, S\tilde{M}_1^{\e})$. Hence, 
\begin{equation*}
    \hat{\Phi} : D_1 \cup \sN(\omega_1, S\tilde{M}_1^{\e}) \to S\tilde{M}_2^{\e}
\end{equation*}
is real-analytic. Since $\hat{\Phi}$ is a local isometry on $S\tilde{U}$, it follows by analytic continuation that it is a local isometry everywhere on $D_1 \cup \sN(\omega_1, S\tilde{M}_1^{\e})$. But this means that $\omega_1$ is an interior point of $D_1$, contradicting the definition of $\omega_1$. Therefore, we must indeed have that $SN(\epsilon_0) \cup D_1 = S \tilde{M}^{\e}_1$, as desired. In particular, as we can make $\epsilon_0$ arbitrarily small, it follows that $\sB_1(S\tilde{M}_1^{\e}) = \sB_2(S\tilde{M}_2^{\e})$. Therefore, we have a real-analytic bijection $\hat{\Phi} : S\tilde{M}_1^{\e} \to S\tilde{M}_2^{\e}$, which is a local isometry everywhere. Hence it is an isometry. This complete the proof of Theorem \ref{isometry of sphere bundles theorem}.
\end{proof}

Unlike the argument used in \cite{Krupchyk2011}, we still have to show that the the isometry $\hat{\Phi} : S\tilde{M}_1^{\e} \to S\tilde{M}_2^{\e}$ descends to an isometry of the underlying base manifolds. To this end, first note that $\hat{\Phi}$ is fiber-preserving on $\tilde{U}$. That is, on $\tilde{U}$, it takes the form
\begin{equation*}
    \hat{\Phi}(x,\omega) = \big( \phi(x), \tilde{\Phi}(x,\omega) \big),
\end{equation*}
where $\phi : \tilde{U} \to \tilde{U}$ is the real-analytic isometry that identifies $(\tilde{U}, g_1)$ and $(\tilde{U}, g_2)$. We have previously identified these two sets implicitly and taken $\phi$ to be the identity, but here we make the identifying map explicit. So, if $\pi_i : S \tilde{M}_i \to \tilde{M}_i$ are the canonical projections, then we have
\begin{equation*}
    (\pi_2 \circ \hat{\Phi})|_{\tilde{U}} = \phi \circ \pi_1.
\end{equation*}
The real-analyticity of $\hat{\Phi}$ means that $\pi_2 \circ \hat{\Phi}$ will be constant along fibres everywhere on $S \tilde{M}_2^{\e}$, and therefore there is a well-defined real-analytic map $\Phi : \tilde{M}_1^{\e} \to \tilde{M}_2^{\e}$ such that
\begin{equation*}
    \pi_2 \circ \hat{\Phi} = \Phi \circ \pi_1
\end{equation*}
everywhere on $\tilde{M}_1^{\e}$, with $\Phi|_{\tilde{U}} = \phi$. That $\Phi$ is real-analytic, bijective and an isometry immediately follows from the corresponding properties of $\hat{\Phi}$. This proves Theorem \ref{main reconstruction theorem}. \\

\begin{remark}\label{remark: no local theorem}
    We can now appreciate why working with the Beltrami field equation poses a difficulty in extending the above proof of Theorem \ref{main reconstruction theorem} to the case of local data. In fact, the difficulty seems to lie entirely with the proof of the important but technical Lemma \ref{lemma sequences}. In the case of local data for the Hodge Laplacian, the analogue result \cite[Lemma 3.8]{Krupchyk2011} is proven using the vanishing of the Green's form on the boundary of the extended manifold, in conjunction with analytic continuation. The boundary values of the $b$-fields, however, do not appear in our proof of Theorem \ref{main reconstruction theorem}. This is for two main reasons: first, because the $b$-fields are generated by interior loops of radius less than $\e$ as opposed to point singularities, we do not reconstruct the initial extended manifold $\tilde{M}_i$ on which the $b$-fields are defined, but rather the slightly smaller manifold $\tilde{M}_i^{\e}$. Moreover, in the case of local data, it is the point $\x_0$ appearing in Equation \eqref{lemma sequences important equation} that lies on the boundary of $S \tilde{M}^{\e}_{2}$, rather than the evaluation point $y$; we have no way, therefore, of incorporating information about the boundary data of the $b$-fields into the proof. Indeed, our proof of Lemma \ref{lemma sequences} relies instead on the $b$-fields agreeing {\em everywhere} in a neighbourhood of the boundary of $\tilde{M}^{\e}_i$.
\end{remark}

\section*{Acknowledgements}

This work has received funding from the European Research Council (ERC) under the European Union's Horizon 2020 research and innovation programme through the grant agreement~862342. The authors are also partially supported by the grants CEX2023-001347-S, RED2022-134301-T and PID2022-136795NB-I00.

%\printbibliography

%\bibliographystyle{alpha} % Elige el estilo deseado (plain, unsrt, alpha, etc.)
%\bibliography{mybibV3.bib}

\end{document}